\documentclass{amsart}
\usepackage{graphicx}
\usepackage[export]{adjustbox}
\graphicspath{ {./images/} }
\usepackage{amssymb}
\usepackage{amsmath}
\usepackage{float}
\usepackage{caption}
\usepackage{hyperref}
\usepackage[dvipsnames]{xcolor} 
\usepackage{tikz-cd}
\usepackage{enumitem}
\DeclareMathOperator{\dom}{dom}

\DeclareMathOperator{\rge}{rge}

\DeclareMathOperator{\cf}{cf}
\DeclareMathOperator{\crit}{crit}

\DeclareMathOperator{\tp}{top}
\DeclareMathOperator{\Ult}{Ult}

\DeclareMathOperator{\stem}{stem}
\theoremstyle{plain}
\newtheorem{thm}{Theorem}[section]
\newtheorem{lemma}[thm]{Lemma}
\newtheorem{prop}[thm]{Proposition}
\newtheorem{claim}{Claim}[thm]
\newtheorem{assumpt}[thm]{Assumption}

\newtheorem{coll}[thm]{Corollary}

\theoremstyle{definition}
\newtheorem{defn}[thm]{Definition}
\newtheorem{rmk}[thm]{Remark}

\usepackage[dvipsnames]{xcolor}
\newenvironment{claimproof}[1]{\par\noindent{Proof}\space#1}{\hfill $\blacksquare$}
\tikzset{close/.style={near start,outer sep=-1pt}}

\title{Another method to add a closed unbounded set of former regulars}
\author{Moti Gitik}
\address{School of Mathematical Sciences, Tel Aviv University, Tel Aviv-Yafo,  Tel Aviv, Israel, 6997801}
\email{gitik@tauex.tau.ac.il}
\author{Sittinon Jirattikansakul}
\address{School of Mathematical Sciences, Tel Aviv University, Tel Aviv-Yafo,  Tel Aviv, Israel, 6997801}
\email{jir.sittinon@gmail.com}
\thanks{The authors were partially supported by ISF grant No. 1216/18.
The first author is grateful to Carmi Merimovich for many helpful discussions on the subject.}

\begin{document}
\maketitle

\begin{abstract}

A club consisting of former regulars is added to an inaccessible cardinal, without changing cofinalities outside it.
The initial assumption is optimal.
A variation of the Radin forcing without a top measurable cardinal is introduced for this.

\end{abstract}

\section{introduction}

From the assumption $\circ(\kappa)=\kappa^+$, by Mitchell,  Radin forcing \cite{michell1982howweakis} adds a club which consists of ground model regular cardinals, while maintaining inaccessibility at $\kappa$.
The large cardinal assumption for that is not optimal.
Later, the first author \cite{gitik1999onclosedunboundedsets} constructed a forcing which adds a club subset of $\kappa$ containing regular cardinals in $V$, while maintaining inaccessibility of $\kappa$, from the optimal assumption.
The forcing was done in two steps.
The first step is to build a $\kappa$-c.c. forcing iteration which is a cardinal-preserving forcing $\mathbb{P}_0$ ascertaining that in $V^{\mathbb{P}_0}$, the set $E:=\{ \alpha<\kappa \mid \alpha$ is regular in $V\}$ is a {\em fat set}, namely for every $\delta<\kappa$ and every club subset $C$ of $\kappa$, there is a closed subset of $C\cap E$ of order-type $\delta$.
Then the second step was done by shooting a club subset of $E$ using closed bounded subsets of $E$.
By fatness, the forcing $\mathbb{P}_1$ is $(\kappa,\infty)$-distributive.
Hence, the forcing produces a  club subset of $\kappa$, preserves all cardinals, preserves inaccessibility of $\kappa$.
However, $V$-regular cardinals even outside the club added by the forcing maybe singularized.
Here we prove the sharper result.
We state our result as the following.

\begin{thm}
\label{mainthm}
Let $\kappa$ be a strongly inaccessible cardinal.
Assume there are a set $X \subseteq \kappa$, $\circ:X \to \kappa$ ($\circ$ is the function representing Mithchell orders), and $\vec{U}$ such that

\begin{enumerate}

\item for $\alpha \in X$, $\circ(\alpha)<\alpha$.

\item for $\alpha \in X$, $\vec{U}(\alpha)=\{U(\alpha,\beta) \mid \beta<\circ(\alpha)\}$ is a Mitchell increasing sequence of normal measures.

\item for $\nu<\kappa$, $\{\alpha \in X \mid \circ(\alpha) \geq \nu\}$ is stationary.

\end{enumerate}

Then there is a cardinal-preserving extension of size $\kappa$ which adds a  club subset $C$ of $\kappa$ containing only regular cardinals in $V$, $\kappa$ is strongly inaccessible in the forcing extension.
Furthermore, $V$-regular cardinals in $\kappa \setminus \lim(C)$ remain regular in the extension.

\end{thm}

In this paper, we present a different method to add a  club subset of $\kappa$ while maintaining inaccessibility of $\kappa$.
Although we work in detail where the assumption  (as in Assumption \ref{initialassumpt1}) is slightly stronger than the assumption in Theorem \ref{mainthm}, one can easily modify  our detail to obtain the same result from the optimal assumption, namely the assumption in Theorem \ref{mainthm}.

The organization of this paper is the followings.

\begin{itemize}

\item In Section \ref{basicfacts} we give an initial assumption, give some basic facts regarding large cardinals and forcings.

\item In Section \ref{forcingbasic} and \ref{forcingPalpha} we define the forcing $P_\alpha$ for $\alpha<\kappa$.
The forcing will approximate the main forcing $\mathbb{P}_\kappa$ in the end. 
We show basic properties and discuss on quotients.

\item In Section \ref{theprikryproperty} we prove the Prikry property for $P_\alpha$.

\item In Section \ref{cardinalbehave} we investigate the cardinal behavior after forcing with $P_\alpha$. 
We also investigate the canonical sets $C_\alpha$ and $C_{\alpha/Q}$ derived from generic objects.
These are served as the initial segments of the final club subset of $\kappa$.

\item In Section \ref{mainforcing} we define the final forcing $\mathbb{P}_\kappa$ and show that it gives the main result.

\item In Section \ref{optimal} we show that the assumption in Theorem \ref{mainthm} is optimal.

\end{itemize}

We assume the readers are familiar with forcings, Radin-Magidor forcing from a coherent sequence of measure (for example, the version that appear in Section 5.2 of Gitik's chapter in \cite{handbook}, iterated forcings, and quotient forcings, although we attempt to present in significant details.

\section{Basic facts}
\label{basicfacts}

We start by stating our initial assumption, which we use from Section \ref{basicfacts} to Section \ref{mainforcing}.

\begin{assumpt}
\label{initialassumpt1}
Assume GCH and $\kappa$ is inaccessible.
Assume there is a set $X \subseteq \kappa$, $\circ:X \to \kappa$, and $\vec{U}$ such that

\begin{enumerate}

\item for $\alpha \in X$, $0<\circ(\alpha)<\alpha$.

\item \label{coherence} $\vec{U}=\{U(\alpha,\beta) \mid \alpha \in X, \beta< \circ(\alpha)\}$ is a coherent sequence of normal measures.

\item for $\nu<\kappa$, $\{\alpha \in X \mid \circ(\alpha) \geq \nu\}$ is stationary.

\end{enumerate}

\end{assumpt}

We remind readers the definition of a coherent sequence of measures.
Define $(\alpha_0,\beta_0)<(\alpha_1,\beta_1)$ for each ordered pairs $(\alpha_0,\beta_0)$ and $(\alpha_1,\beta_1)$ if $\alpha_0<\alpha_1$ or ($\alpha_0=\alpha_1$ and $\beta_0<\beta_1$).
Let $\vec{U} \restriction \gamma=\{U(\alpha,\beta)\in \vec{U} \mid \alpha<\gamma\}$ and $\vec{U} \restriction (\alpha,\beta)=\{U(\alpha^\prime,\beta^\prime) \mid (\alpha^\prime,\beta^\prime)<(\alpha,\beta)\}$.
Coherence states that $j:V \to \Ult(V,U(\alpha,\beta))$, then $j(\vec{U}) \restriction (\alpha+1)=\vec{U} \restriction (\alpha,\beta)$. 
It is known in \cite{gitik1999onclosedunboundedsets} that Assumption \ref{initialassumpt1} is sufficient to construct a forcing which adds a club subset of $\kappa$ of ground model regular cardinals, while maintaining inaccessibility of $\kappa$.
We note again that Assumption \ref{initialassumpt1} is slightly stronger than the assumption in Theorem \ref{mainthm}, which is known to be optimal to build a forcing to add a club subset of $\kappa$ containing $V$-regular cardinals, and $\kappa$ is inaccessible in the extension. 
We can construct a forcing that gives the same effect using only the assumption in Theorem \ref{mainthm}.

\textbf{Notations and conventions}: From this point to Section \ref{mainforcing}, we assign some notations and agreements for our convenience.
Some of them are not standard notations.
For each pair of sets $A$ and $B$, we denote $A \sqcup B$ the usual union of $A$ and $B$, whenever $A \cap B= \emptyset$.
We assume that all measure-one sets concentrate on inaccessible cardinals.
Define $\circ(\gamma)=0$ if and only if $\gamma$ is inaccessible and not measurable.
We also assume that $X$ includes all points $\gamma$ with $\circ(\gamma)=0$, so that $X=\{\alpha<\kappa \mid \alpha$ is inaccessible$\}$.
For $\alpha \in X$ with $\circ(\alpha)>0$, let $\vec{U}(\alpha)=\{U(\alpha,\beta) \mid \beta< \circ(\alpha)\}$ and we say that $A \in \vec{U}(\alpha)$ if and only if $A \in \cap_{\beta<\circ(\alpha)} U(\alpha,\beta)$.
If $\circ(\alpha)=0$, $\vec{U}(\alpha)=\{\emptyset\}$. 
We assume that each $A \in \vec{U}(\alpha)$ is written as a disjoint union of $\langle A_\beta \mid \beta<\circ(\alpha) \rangle$ where $A_\beta \in U(\alpha,\beta)$ and for $\gamma \in A_\beta$, $\circ(\gamma)=\beta$.
For each $A \in \vec{U}(\alpha)$ and $\gamma \in A$, we define $A \restriction \gamma=\{\xi \in A \mid \xi<\gamma$ and $\circ(\xi)<\circ(\gamma)\}$ and $A \setminus \gamma=\{\xi \in A \mid \xi>\gamma\}$.
We note that if $\gamma \in A \in \vec{U}(\alpha)$ and $\circ(\gamma)=0$, then $A \restriction \gamma=\emptyset$.
Assume that for $\gamma \in A$, $A \restriction \gamma \in \vec{U}(\gamma)$.
For $P$-name of ordinal $\dot{\nu}$, define $\bar{\nu}=\sup\{\nu \mid \exists p \in P (p\Vdash \dot{\nu}=\check{\nu})\}+1$.
For each function $F$ with domain $A$ of ordinals, and $\gamma$ is an ordinal, let $F \setminus \gamma=F \restriction (A \setminus \gamma)$.
When we force with a forcing notion $P_\alpha$, then the corresponding relations and other relevant objects may be distinguished by the subscript $\alpha$, e.g. $\leq_\alpha$ and $\leq^*_\alpha$, $\Vdash_\alpha$, and so on.
If $\dot{\nu}$ is a $Q$-name of an ordinal which is at least $\alpha$, we put subscripts $\dot{\nu}/Q$ for the forcing $P_{\dot{\nu}}/Q$ e.g. $\leq_{\dot{\nu}/Q}$ and $\leq^*_{\dot{\nu}/Q}$.
 If $Q=P_\alpha$, then we abbreviate the subscript $\dot{\nu}/Q$ by $\dot{\nu}/\alpha$.
 We remove subscripts sometimes if they are clear from the context.
We often omit the check symbols for the names representing ground model elements in a forcing (except in Section \ref{forcingbasic} where we emphasize the places where names are used).
Fix a limit ordinal $\gamma$.
A {\em club} in $\gamma$ is a closed unbounded subset of $\gamma$.
Even though $\cf(\gamma)=\omega$, a cofinal subset of $\gamma$ of order-type $\omega$ is also considered as a club. 

We start off in this section by recalling some known facts, which will be used in further sections.

\begin{prop}
\label{integrate}
Let $A \in \vec{U}(\alpha)$.
For $\gamma \in A$, let $B_\gamma \in \vec{U}(\gamma)$ and $B_\gamma \subseteq A \restriction \gamma$.
Then there is $A^* \in \vec{U}(\alpha)$, such that $A^* \subseteq A$, and for $\gamma \in A^*$, $A^* \restriction \gamma \subseteq B_\gamma$.

\end{prop}

\begin{proof}

For each $\beta<\circ(\alpha)$, we define $X_1^\beta \cup X_2^\beta \cup X_3^\beta$ as follows.
Let $X_1^\beta=j_{U(\alpha,\beta)}(\gamma \mapsto B_\gamma)(\alpha)$.
Then $X_1^\beta \in j_{U(\alpha,\beta)}(\vec{U})(\alpha)=\vec{U}(\alpha) \restriction \beta$.
Also, $X_2^\beta:=\{\gamma \in A \mid \circ(\gamma)=\beta$ and $X_1^\beta \cap \gamma=B_\gamma\} \in U(\alpha,\beta)$.
Finally, let $X_3^\beta=\{\gamma \in A \mid \circ(\gamma)>\beta\}$.
Then $X_3^\beta \in \vec{U}(\alpha) \setminus (\beta+1)$.
Let $A^*=\cap_{\beta<\circ(\alpha)} (X_1^\beta \cup X_2^\beta \cup X_3^\beta)$.
By $\alpha$-completeness of measures, $A^* \in \vec{U}(\alpha)$.
Furthermore, we have that for $\gamma \in A^*$ with $\circ(\gamma)=\beta$, $A^* \restriction \gamma \subseteq X_1^\beta \cap \gamma \subseteq B_\gamma$.

\end{proof}

\begin{prop}
\label{measureoneinforcing}
Let $P \in V_{\xi+1}$ be a forcing notion and $G$ be $P$-generic.
Then there is a coherent sequence of measures $\langle U(\alpha,\beta) \mid \beta <\circ(\alpha), \alpha>\xi,$ and $ \alpha \in X \rangle$, each $U^\prime(\alpha,\beta)$ extends $U(\alpha,\beta)$.

\end{prop}

\begin{proof}

In $V[G]$, for $\rho>\xi$, define $U^\prime(\rho,\gamma)=\{A \subseteq \rho \mid \exists B \in U(\rho,\gamma) (B \subseteq A)\}$.
For each ordered pair $(\alpha,\beta)$, $\alpha>\xi$, let $j:=j_{U(\alpha,\beta)}: V \to M$.
Since $P \in V_{\xi+1}$ and $\alpha>\xi$ is inaccessible, $j$ extends to $j:V[G] \to M[G]$.
Since $\dom(\vec{U}^\prime)=\{(\rho,\xi) \in \dom(\vec{U}) \mid \rho>\xi\}$, $\dom(j(\vec{U}^\prime)) \restriction \alpha+1=\dom(\vec{U}^\prime) \restriction (\alpha,\beta)$.
Since $\crit(j)=\alpha$, $M[G] \subseteq V[G]$, and $V_{\alpha+1}^{M[G]}=V_{\alpha+1}^{V[G]}$,
by following the coherence of $\vec{U}$, we have that $j(\vec{U}^\prime) \restriction \alpha+1=\vec{U}^\prime \restriction (\alpha,\beta)$.

\end{proof}

Observe that in Proposition \ref{measureoneinforcing}, $U(\rho,\xi)$ is {\em dense} in $U^\prime(\rho,\xi)$: $U(\rho,\xi) \subseteq U^\prime(\rho,\xi)$ and for $A \in U^\prime(\rho,\xi)$, there is $B \in U(\rho,\xi)$ such that $B \subseteq A$.

\begin{prop}
\label{measureinground}
Let $P \in V_\alpha$ and $\alpha \in X$.
If $\dot{A}$ is a $P$-name of an element in $\vec{U}(\alpha)$.
Then there is $B \in \vec{U}(\alpha)$ such that $\Vdash_P \check{B} \subseteq \dot{A}$.
\end{prop}

\begin{proof}

Take $B=\cap \{A^\prime \mid \exists p \in P (p \Vdash_P \dot{A}=\check{A}^\prime \rangle\}$.
Then $B \in \vec{U}(\alpha)$ since $|P|<\alpha$.

\end{proof}

We end this section by recalling that if for $\gamma<\alpha$, $A_\gamma \in \vec{U}(\alpha)$, the {\em diagonal intersection of $\langle A_\gamma \mid \gamma<\alpha \rangle$}, denoted by $\triangle_{\gamma<\alpha} A_\gamma$, is the set $\{\gamma<\alpha \mid \forall \xi<\gamma (\gamma \in A_\xi)\}$.
It is easy to check that $\triangle_{\gamma<\alpha}A_\gamma \in \vec{U}$.
A variation of the diagonal intersection is also defined: for $\vec{\gamma} \in [\alpha]^n$, let $A_{\vec{\gamma}} \in \vec{U}(\alpha)$, we define $\triangle_{\vec{\gamma}} A_{\vec{\gamma}}=\triangle_{\xi<\alpha} \cap_{\{\vec{\gamma} \mid \max(\vec{\gamma})=\xi\} } A_{\vec{\gamma}}$. 

\section{The forcing $\mathbb{P}_\alpha$: basic cases}
\label{forcingbasic}

In Section \ref{forcingbasic} and \ref{forcingPalpha} we build forcings $P_\alpha$ for $\alpha \in X$ by recursion.
In this section we define $P_\alpha$ for the first two ordinals in $X$.
The reason we define forcings at two levels is that we want to introduce quotient forcings, which are more meaningful when we consider two levels.
All the forcings and the quotients in this section are equivalent to the trivial forcing.
However, the construction of $P_\alpha$ for $\alpha \in X$ and their quotients which appear in both sections will satisfy a long list of the properties, which will be verified inductively.
We also introduce all necessary notions which will appear in the list of properties, which will be elaborated later in Proposition \ref{theproperties}.

\begin{defn}{\textbf{(The first basic case)}}

Let $\alpha_0=\min(X)$.

\begin{enumerate}

\item Define $P_{\alpha_0}$ as $\{1_{\alpha_0}$, $p_0 \}$ where $p_0=\langle \alpha_0,\emptyset,\emptyset \rangle$.
Define the forcing relation $\leq$ on $P_{\alpha_0}$ so that $p_0$ is stronger than $1_{\alpha_0}$.
Let $\leq^*=\leq$.
We call  $p_0$ a {\em pure} condition.
There are no {\em impure} conditions.
The {\em stems} for both conditions are empty.
Suppose $G$ is $P_{\alpha_0}$-generic, then $G=P_{\alpha_0}$. 
In $V[G]$, define $C_{\alpha_0}=\{\alpha_0\}$, and let $\dot{C}_{\alpha_0}$ be the $P_{\alpha_0}$-name of such $C_{\alpha_0}$. 

\item Assume $G$ is $P_{\alpha_0}$-generic, define $(P_{\alpha_0}[G], \leq, \leq^*)$ as the trivial forcing $(\{\emptyset\},\leq,\leq^*)$ in $V[G]$.
Let $P_{\check{\alpha}_0}/P_{\alpha_0}$ be a $P_{\alpha_0}$-name of the forcing $P_{\alpha_0}[G]$, where $\check{\alpha}_0$ is the $P_{\alpha_0}$-name.
If $H$ is $P_{\alpha_0}[G]$-generic, in $V[G][H]$, define $C_{\alpha_0/\alpha_0}=\emptyset$, and let $\dot{C}_{\check{\alpha}_0/\alpha_0}$ be a $P_{\alpha_0}* P_{\check{\alpha}_0}/P_{\alpha_0}$-name for such $C_{\alpha_0/\alpha_0}$.

\item If $G$ is $P_0$-generic, define $P_{\check{\alpha}_0}/P_{\alpha_0}[G]$ as $P_{\alpha_0}[G]$.

\end{enumerate}
\end{defn}

\begin{defn}{\textbf{(The second basic case)}}
\label{basicsecondcase}

Let $\alpha_0=\min(X)$ and $\alpha_1=\min(X \setminus (\alpha_0+1))$.

\begin{enumerate}

\item Define $P_{\alpha_1}$ as $\{1_{\alpha_1},p_1,p_{01}\}$,  $p_1=\langle \alpha_1,\emptyset,\emptyset\rangle$, $p_{01}=\langle \alpha_0,\emptyset,\emptyset,\check{\alpha}_0,\check{\emptyset}\rangle^\frown \langle \alpha_1,\emptyset,\emptyset \rangle$ where the checks are considered as a $P_{\alpha_0}$-name.
Define the forcing relation $\leq$ on $P_{\alpha_1}$ so that $1_{\alpha_1}$ is the weakest condition, and $p_1$ and $p_{01}$ are not comparable.
Let $\leq^*=\leq$.
We call $p_1$ a {\em pure} condition and $p_{01}$ an {\em impure} condition.
The {\em stems} for $1_{\alpha_1}$ and $p_1$ are empty, and the stem for $p_{01}$ is $\langle \alpha_0,\emptyset,\emptyset,\check{\alpha}_0,\check{\emptyset} \rangle$.
If $G$ is $P_{\alpha_0}$-generic, then either $G=\{1_{\alpha_1},p_1\}$, which we define the corresponding $C_{\alpha_1}=\{\alpha_1\}$, otherwise, $G=\{1_{\alpha_1},p_{01}\}$, in which case, we define $C_{\alpha_1}=\{\alpha_0,\alpha_1\}$.
Let $\dot{C}_{\alpha_1}$ be a $P_{\alpha_1}$-name for such $C_{\alpha_1}$.

\item Assume $G$ is $P_{\alpha_0}$-generic, define $P_{\alpha_1}[G]$ as $\{1_{\alpha_1}, \langle \alpha_0, \emptyset \rangle^\frown \langle \alpha_1,\emptyset,\emptyset \rangle\}$.
We consider $1_{\alpha_1}$ as the weakest condition and $\leq^*=\leq$.
If $H$ is $P_{\alpha_1}[G]$-generic, define $C_{\alpha_1/\alpha_0}=\{\alpha_1\}$, and let $\dot{C}_{\check{\alpha}_1/\alpha_0}$ be a $P_{\alpha_0}* P_{\check{\alpha}_1/\alpha_0}$-name for such $C_{\alpha_1/\alpha_0}$.
We do not define pure, impure conditions, and stems in this case.
Define $P_{\check{\alpha}_1}/P_{\alpha_0}$ as the $P_{\alpha_0}$-name of the forcing $P_{\alpha_1}[\dot{G}]$ where $\dot{G}$ is the canonical name for generic for $P_{\alpha_0}$, where the check is the $P_{\alpha_0}$-name.

\item Assume $G$ is $Q:=P_{\alpha_0} * P_{\check{\alpha}_0}/P_{\alpha_0}$-generic where the check is the $P_{\alpha_0}$-name, define $P_1[G]=\{1_{\alpha_1},\langle \alpha_1,\emptyset,\emptyset \rangle\}$ (in $V[G]$), with the induced ordering in $P_1$, i.e., $1_{\alpha_1}$ is the weakest element.
Define $\langle \alpha_1,\emptyset,\emptyset \rangle[G]=\langle \alpha_1,\emptyset,\emptyset \rangle$.
The condition $\langle \alpha_1,\emptyset,\emptyset \rangle$ is a {\em pure} condition.
The {\em stems} of both conditions are empty.
There are no impure conditions.
Define $P_{\alpha_1}/Q$ as the $Q$-name of the forcing $P_{\alpha_1}[\dot{G}]$ where $\dot{G}$ is the canonical name for generic for $Q$.
Define $C_{\alpha_1/Q}=\{\alpha_1\}$.
Let $\dot{C}_{\alpha_1/Q}$ be a $P_Q * P_{\check{\alpha}_1/Q}$ be a name for such $C_{\alpha_1/Q}$.

\item Assume $G$ is $P_{\alpha_1}$-generic, let $P_{\alpha_1}[G]$ as the trivial forcing $(\{\emptyset\},\leq,\leq^*)$ in $V[G]$.
Define $P_{\check{\alpha}_1}/P_{\alpha_1}$ (where the check is the $P_{\alpha_1}$-name) as the $P_{\alpha_1}$-name of $P_{\alpha_1}[G]$.
If $H$ is $P_{\alpha_1}[G]$-generic, in $V[G][H]$, define $C_{\alpha_1/\alpha_1}=\emptyset$.
Let $\dot{C}_{\check{\alpha_1}/{\alpha_1}}$ be a $P_{\alpha_1} * P_{\check{\alpha}_1}/P_{\alpha_1}$-name for such $C_{\alpha_1/\alpha_1}$.

\item For $i=0,1$, if $G$ is $P_i$-generic, define $P_{\check{\alpha}_1}/P_{\alpha_i}[G]=P_{\alpha_1}[G]$.

\end{enumerate}

Note that $p_{01}$ in Definition \ref{basicsecondcase} is considered in the following manner.
First we consider the first tuple  of $p_{01}$, which is $\langle \alpha_0,\emptyset,\emptyset,\check{\alpha}_0,\check{\emptyset} \rangle$.
The first three entries $\langle \alpha_0,\emptyset,\emptyset \rangle$ is an element in $P_{\alpha_0}$. 
The fourth entry $\check{\alpha}_0$ is considered as a $P_{\alpha_0}$-name of an ordinal in the interval $[\alpha_0,\alpha_1) \cap X$, which is just $\{\alpha_0\}$.
It serves as a parameter in the quotient forcing the fifth entry is living in.
In this case, it is $P_{\check{\alpha}_0}/P_{\alpha_0}$, where the denominator is the first entry of the tuple. 
The last entry $\check{\emptyset}$ is a $P_{\alpha_0}$-name of the element in $P_{\check{\alpha}_0}/P_{\alpha_0}$.
Hence, $\langle \alpha_0,\emptyset,\emptyset,\check{\alpha}_0,\check{\emptyset} \rangle$ is regarded as a condition in $P_{\alpha_0} * P_{\check{\alpha}_0}/P_{\alpha_0}$.
When $G$ is $P_{\alpha_0}$-generic, we  consider $P_{\alpha_1}[G]$ as the ``leftover" of the conditions in $P_{\alpha_1}$ whose initial segments are in $G$. 
For example, the condition $p_{01}=\langle \alpha_0,\emptyset,\emptyset, \check{\alpha}_0,\check{\emptyset} \rangle^\frown \langle \alpha_1,\emptyset,\emptyset \rangle$ can be factored to $\langle \alpha_0,\emptyset,\emptyset \rangle \in P_{\alpha_0}$, and the rest: $\langle \check{\alpha}_0,\check{\emptyset}\rangle^\frown \langle \alpha_1,\emptyset,\emptyset \rangle$.
If $G$ is $P_{\alpha_0}$-generic, then the non-trivial element in $P_{\alpha_1}[G]$ is $\langle \check{\alpha}_0,\check{\emptyset}\rangle ^\frown \langle \alpha_1,\emptyset,\emptyset \rangle[G]=\langle \alpha_0,\emptyset\rangle ^\frown \langle \alpha_1,\emptyset,\emptyset \rangle$ which is considered as a condition in $P_{\alpha_0}[G] \times P_{\alpha_1}[G]$.

\end{defn}

We now list all the properties which we need to be true in the more general settings.
Stronger versions of some properties are true, but we state in a way that is aligned together with properties in general cases, i.e. Proposition \ref{theproperties}.
Let $\alpha_0$ and $\alpha_1$ be as in Definition \ref{basicsecondcase}.
Let $0 \leq i \leq j \leq 1$.
\begin{itemize}

\item $(P_{\alpha_0},\leq,\leq^*)$ and $(P_{\alpha_1},\leq,\leq^*)$ have the Prikry property , preserve all cardinals, and GCH.
$\Vdash_{\alpha_i} (P_{\check{\alpha}_j}/P_{\alpha_i},\leq,\leq^*)$ has the Prikry property, preserves all cardinals, and GCH, and $\Vdash_{\alpha_i} P_{\check{\alpha}_i}/P_{\alpha_i}$ is the trivial forcing.

\item $|P_{\alpha_0}| \leq \alpha_0^+$, $|P_{\alpha_1}| \leq \alpha_1^+$ (they are actually finite). 
Hence, they are $\alpha_0^{++}$-c.c., and $\alpha_1^{++}$-c.c., respectively.
$\leq_{\alpha_0}^*$ is $\alpha_0$-closed and $\leq_{\alpha_1}^*$ is $\alpha_1$-closed (both are $\infty$-closed).
Furthermore, $\Vdash_{\alpha_i} (P_{\check{\alpha}_j}/P_{\alpha_i}, \leq^*)$ is $\min(X \setminus (\alpha_i+1))-$closed (it is $\infty$-closed).
It is easy to check that $\Vdash_{\alpha_i} |P_{\check{\alpha}_j}/P_{\alpha_i}| \leq \alpha_j^+$ (it is finite), hence $\alpha_j^{++}$-c.c., and $\leq^*$ is $\alpha_j$-closed (it is $\infty$-closed).
Hence, by counting nice names, the two-step iteration $P_{\alpha_i} * P_{\check{\alpha}_j}/P_{\alpha_i}$ has size $\alpha_j^+$ in $V$, hence $\alpha_j^{++}$-c.c., has the Prikry property, preserves all cardinals, and GCH.

\item Consider the condition $p_{01} \in P_{\alpha_1}$ as in Definition \ref{basicsecondcase}.
Set $P^s=P_{\alpha_0} * P_{\check{\alpha}_0}/P_{\alpha_0}$.
Let $G$ be $P^s$-generic (it is unique).
Then, in $V[G]$, $P_{\alpha_1}[G]$ has the Prikry property, preserves all cardinals, GCH, $|P_{\alpha_1}[G]| \leq \alpha_1^+$, so $\alpha_1^{++}$-c.c.
Notice that $p_{01}=s^\frown p_1$, where $s \in G$, and $p_1=p_1[G] \in P_{\alpha_1}[G]$.

\item If $p \in P_{\alpha_i}$, $\gamma<\alpha_i$, $\langle p_\beta \mid \beta< \gamma\rangle$ is a sequence of conditions such that $p_\beta \leq^*_{\alpha_i} p$ for all $\beta$, then there is $p^* \leq^*_{\alpha_i} p$ such that $p^* \leq^*_{\alpha_i} p_\beta$ for all $\beta$.
If $\Vdash_i \dot{q} \in P_{\check{\alpha}_j}/P_{\alpha_i}$, $\gamma<\min(X \setminus (\alpha_i+1))$, and for $\beta<\gamma$, there are $p_\beta, \dot{q}_\beta$ such that $p_\beta \in P_{\alpha_i}$, $p_\beta \Vdash_i \dot{q}_\beta \leq^* \dot{q}$, then there is $\dot{q}^*$ such that $\Vdash \dot{q}^* \leq^* \dot{q}$ and for all $\beta$, $p_\beta \Vdash \dot{q}^* \leq^* q_\beta$.

\item Let $p_{01} \in P_{\alpha_1}$ be as before.
Set $Q=P_{\alpha_0} * P_{\check{\alpha}_0}/P_{\alpha_0}$ and
$s:=\stem(p_{01})=\langle \alpha_0,\emptyset,\emptyset,\check{\alpha}_0,\check{\emptyset} \rangle \in Q$. 
Set $P^s=P_{\alpha_0} * P_{\check{\alpha}_0}/P_{\alpha_0}$.
Let $G$ be $P^s$-generic, and $H$ be $P_{\alpha_1}[G]$-generic containing $\langle \alpha_1,\emptyset,\emptyset \rangle$.
By setting $I=\{1_{\alpha_1},p_{01} \}$, we have that $I$ is $P_{\alpha_1}$-generic, $V[G][H]=V[I]$.
Conversely, assume $I$ is $P_{\alpha_1}$-generic containing $p_{01}$.
Let $G$ be $P^s$-generic, and $H$ be $P_1[G]$-generic (these are unique), then $V[G][H]=V[I]$.
Notice that $p_{01}=p_{01} \restriction Q {}^\frown p_1$ where $p_{01} \restriction Q=\stem(p_{01}) \in Q$, and $p_1=p_1[G] \in P_{\alpha_1}[G]$.

\end{itemize}

\begin{defn}

Let $Q=P_{\alpha_0}$ or $P_{\alpha_0}* P_{\check{\alpha}_0}/P_{\alpha_0}$ and $G$ be $Q$-generic.
Define $\Vdash_{P_{\alpha_1}}[G] \dot{q} \in P_{\check{\alpha}_1}/P_{\alpha_1}[G]$ iff $\dot{q} \in (P_{\check{\alpha}_1})[\dot{I}]$ where $\dot{H}$ is the name for $P_{\alpha_1}[G]$-generic, and $\dot{I}$ is the corresponding canonical name for $P_{\alpha_1}$-generic, i.e. $\Vdash_{P_{\alpha_1}[G]} \dot{I}=G * \dot{H}$.

\end{defn}

\begin{itemize}

\item If $G$ is $P_{\alpha_0}$-generic, then $C_{\alpha_0}=\{\alpha_0\} \subseteq \alpha_0+1$, $C_{\alpha_0} \setminus \{\alpha_0\}=\emptyset$ is a closed bounded subset of $\alpha_0$.
If $G$ is $P_{\alpha_1}$-generic, then $C_{\alpha_1}$ is either $\{\alpha_1\}$ or $\{\alpha_0,\alpha_1\}$, both are subsets of $\alpha_1+1$.
$C_{\alpha_1} \setminus \{\alpha_0\}$ is either $\emptyset$ or $\{\alpha_0\}$ which are closed bounded subsets of $\alpha_1$.
$C_{\alpha_i/\alpha_i}=\emptyset$.
$C_{\alpha_1/\alpha_0}=\{\alpha_1\} \subseteq \alpha_1+1$, and $C_{\alpha_1/\alpha_0} \setminus \{\alpha_1\}=\emptyset$ is a closed bounded subset of $\alpha_1$.
All regular cardinals are regular in these extensions.

\item Consider $p_{01} \in P_{\alpha_1}$ as before, $G$ is $Q:=P_{\alpha_0} * P_{\check{\alpha}_0}/P_{\alpha_0}$-generic, $H$ is $P_{\alpha_1}[G]$-generic, and $I$ be the canonical $P_{\alpha_1}$-generic such that $V[G][H]=V[I]$.
We see that in $V[I]$, $C_{\alpha_0}=\{\alpha_0\}$, $C_{\alpha_0/\alpha_0}=\{\emptyset\}$, $C_{\alpha_1/Q}=\{\alpha_1\}$.
Note that $I$ is $P_{\alpha_1}$-generic containing $p_{01}$ so that $C_{\alpha_1}=\{\alpha_0,\alpha_1\}=C_{\alpha_0} \sqcup C_{\alpha_0/\alpha_0} \sqcup C_{\alpha_1/Q}$.

\end{itemize}

\section{The forcings $P_\alpha$: the general cases}
\label{forcingPalpha}

We now state a list of properties that we wish to be true in general.
Note that all the properties are true in our basic cases (we did not prove the properties for every possible case, but one can simply apply the arguments from Section \ref{forcingbasic} to check all the cases).

\begin{prop}{\textbf{(The inductive assumptions)}}
\label{theproperties}

Fix $\alpha \in X$, and $P_\alpha$-name $\dot{\nu}$ with $\Vdash_\alpha \dot{\nu} \in X$ and $\dot{\nu} \geq \alpha$.
The following properties hold:

\begin{enumerate}

\item \label{prikry} $(P_\alpha,\leq_\alpha,\leq^*_\alpha)$ has the Prikry property, preserves all cardinals, and GCH.
 $(P_{\dot{\nu}}/P_\alpha, \leq_{\dot{\nu}/\alpha},\leq^*_{\dot{\nu}/\alpha})$ is forced by $P_\alpha$ to have the Prikry property, preserve all cardinals, and GCH.
 $\Vdash_\alpha \dot{\nu}=\alpha$ implies $(P_{\dot{\nu}}/P_\alpha,\leq,\leq^*)$ is the trivial forcing.

\item \label{ccandclosed} $|P_\alpha| \leq \alpha^+$, hence $\alpha^{++}$-c.c. $(\{p \in P_\alpha \mid p$ is pure$\},\leq^*)$ is $\alpha$-closed  $\Vdash_\alpha |P_{\dot{\nu}}/P_\alpha| \leq \dot{\nu}^+$, hence $\dot{\nu}^{++}$-c.c., $(\{p \in P_{\dot{\nu}}/P_\alpha \mid p$ is pure$\},\leq^*)$ is $\dot{\nu}$-closed.

Recall from our notation that $\bar{\nu}=\sup\{\nu \mid \exists p \in P_\alpha(p \Vdash_\alpha \dot{\nu}=\check{\nu})\}+1$.
Note that $\bar{\nu}<\kappa$ since $|P_\alpha|<\kappa$ and $\dot{\nu}$ is forced to be less than $\kappa$.
By the basic facts on two-step iteration, and GCH, the forcing $P_\alpha * P_{\dot{\nu}}/P_\alpha$ has size at most $\bar{\nu}^+$, has $\bar{\nu}^{++}$-c.c., preserves cardinals and GCH.
Furthermore, $(\{ p \in P_\alpha \mid ((\stem(p)_0)_0)=\gamma\},\leq^*)$ is $\gamma$-closed and $\Vdash_\gamma (P_{\dot{\nu}}/P_\alpha, \leq^*)$ is $\min(X \setminus (\alpha+1))$-closed.

\item \label{factor1} Let $Q$ be $P_\beta$ or $P_\beta * P_{\dot{\gamma}}/P_\beta$ where $\beta<\alpha$, $\Vdash_\beta \beta \leq \dot{\gamma}<\alpha$.
We define 
\begin{align*} 
& \Vdash_Q (\dot{q} \in P_\alpha/Q \Leftrightarrow \exists p_0^\frown p_1 \in Q (p_0 \in Q \text{ and } p_1[\dot{G}]=\dot{q}))\\
\text{ and }&  (\Vdash_Q \dot{q} \in P_\alpha/Q) \text{ iff for every } Q\text{-generic } G, \dot{q}[G] \in P_\alpha[G]. 
\end{align*}
 (the notion $P_\alpha[G]$ and $p_1[G]$ will be explained further in Definition \ref{quotient}).
Then $\Vdash_Q (P_\alpha/Q,\leq,\leq^*)$ has the Prikry property,  preserves all cardinals, GCH.
With nice names, $|P_\alpha/Q|^V \leq \alpha^+$, and so $(P_\alpha[G],\leq)$ is $\alpha^{++}$-c.c.
If pure conditions are defined in $P_\alpha[G]$, then $(\{p \in P_\alpha[G] \mid p$ is pure$\},\leq^*)$ is $\alpha$-closed.
 If $Q$ is $P_\alpha$, then $P_\alpha/P_\alpha$ is the trivial forcing.

\item \label{nameclosed} 
Let $p \in P_\alpha$, $\stem(p_0)_0=\gamma$.
Suppose that $\beta<\gamma$ and for all $\xi<\beta$, $p_\xi \leq^*_\alpha p$, then there is $p^* \leq^*_\alpha p$ such that for all $\xi<\beta$, $p^* \leq^*_\alpha p_\xi$.
Furthermore, let $\Vdash_\alpha \dot{q} \in P_{\dot{\nu}}/P_\alpha$.
Let $\beta^\prime<\min(X \setminus (\alpha+1))$.
Suppose that for $\xi<\beta$, there are $p_\xi \in P_\alpha$ and $\dot{q}_\xi$ such that $p_\xi \Vdash_\alpha \dot{q}_\xi \leq^* \dot{q}$.
Then there is $\dot{q}^*$ such that $\Vdash_\alpha \dot{q}^* \leq^* \dot{q}$ and for all $\xi$, $p_\xi \Vdash_\alpha \dot{q}^* \leq^* \dot{q}_\xi$. 

\item \label{factor2} With the notions in (\ref{factor1}), let $H$ be $V[G]$-generic over  $P_\alpha[G]$.
Then there is a canonical generic $I$ over $P_\alpha$ such that $V[G][H]=V[I]$.
$I$ is defined as $\{p \in P_\alpha \mid p$ is of the form $p_0^\frown p_1$, $p_0 \in G$, $p_1[G] \in H\}$.
Conversely, let $Q$ be the forcing $P_\beta$ or $P_{\dot{\gamma}}/P_\beta$ where $\Vdash_\beta \beta \leq \dot{\gamma}<\alpha$.
Let $I$ be $P_\alpha$-generic, $p \in I$, and $p \restriction Q \in Q$ exists.
Define

\begin{center}

$G:=\{p^\prime \restriction Q \mid p^\prime \in I$ and $p^\prime \restriction Q$ exists$\}$

\end{center} 

is $Q$-generic, 

\begin{center}

$H:=\{(p^\prime[G] \mid \exists t \in G(t^\frown  p^\prime \in I)\}$

\end{center} 

is $V[G]$-generic over $P_\alpha[G]$, and $V[G][H]=V[I]$.
This makes $P_\alpha$ to be equivalent to $Q * P_\alpha/Q$.

\item \label{twolevelquotient} With the notions in (\ref{factor1}), $(P_{\dot{\nu}}/P_\alpha)[G]:=P_{\dot{\nu}[G]}[G]/P_\alpha[G]$ is also defined ($\dot{\nu}[G]$ is a $P_\alpha[G]$-name) as follows: $\Vdash_{P_\alpha[G]} \dot{q} \in (P_{\dot{\nu}}/P_\alpha)[G]$ iff $\dot{q} \in (P_{\dot{\nu}})[G][\dot{H}]$ where $\dot{H}$ is the canonical name for $P_\alpha[G]$-generic.
If $I$ is $P_\alpha$-generic, we define $P_{\dot{\nu}}[I]/P_\alpha[I]$ as $P_{\dot{\nu}[I]}[I]$, which was defined in (\ref{factor2}).

\item \label{club} The notions $\dot{C}_\alpha$ and $\dot{C}_{\dot{\nu}/\alpha}$ are sets defined from generics over $P_\alpha$ and $P_{\dot{\nu}}/P_\alpha$, respectively.
Let $G$ be $P_\alpha$-generic and $C_\alpha=\dot{C}_\alpha[G]$.
Then $C_\alpha \subseteq \alpha+1$, $\max(C_\alpha)=\alpha$.
If $\circ(\alpha)=0$, then $C_\alpha \setminus \{\alpha\}$ is a closed unbounded subset of $\alpha$.
Otherwise, $C_\alpha \setminus \{\alpha\}$ is a club subset of $\alpha$ and singularizes $\alpha$ to have cofinality $\cf(\omega^{\circ(\alpha)})$ (where the exponentiation is the ordinal one).
If $\lambda$ is regular in $V$ and $\lambda \not \in \lim(C_\alpha \setminus \{\alpha\})$, then $\circ(\lambda)=0$ and $\lambda$ is regular in $V[G]$.
Now, let $H$ be $(P_{\dot{\nu}})[G]$-generic, $\nu=\dot{\nu}[G]$, and $C_{\nu/\alpha}=\dot{C}_{\nu/\alpha}[G][H]$.
Then, $C_{\nu/\alpha} \subseteq (\alpha,\nu]$ so that if  $\nu=\alpha$, $C_{\nu/\alpha}=\emptyset$, otherwise $\max(C_{\nu/\alpha})=\nu$.
Assume $\nu \neq \alpha$, if $\circ(\nu)=0$, $C_{\nu/\alpha} \setminus \{\nu\}$ is a closed bounded subset of $\nu$, otherwise, $C_{\nu/\alpha} \setminus \{ \nu\}$ is a club subset of $\nu$, which singularizes $\nu$ to have cofinality $\cf(\omega^{\circ(\nu)})$ (where the exponentiation is the ordinal one).
If $\lambda$ is regular in $V[G]$ and $\lambda \not \in \lim(C_{\nu/\alpha})$, then $\circ(\lambda)=0$ and $\lambda$ is regular in $V[G][H]$.

\item \label{clubfactor}
Fix the notions as in (\ref{factor1}) and (\ref{factor2}), i.e. $Q$ is either $P_\beta$ or $P_{\dot{\gamma}}/P_\beta$ where $\Vdash_\beta \beta \leq \dot{\gamma}<\alpha$, $G$ is $Q$-generic, $H$ is $P_\alpha[G]$-generic, and $I$ is the canonical generic such that $V[I]=V[G][H]$.
Consider the model $V[I]$.
If $Q=P_\beta$, let $C_\beta$ be derived from $G$ and let $C_Q=C_\beta$.
If $Q=P_\beta * P_{\dot{\gamma}}/P_\beta$, let $C_\beta$ be derived from $G \restriction P_\beta$,  $\gamma=\dot{\gamma}[G \restriction P_\beta]$, $C_{\gamma/\beta}$ be derived from $G \restriction (P_{\dot{\gamma}}[G \restriction \beta])$, and $C_Q=C_\beta \sqcup C_{\gamma/\beta}$.
Then there is a set $C_{\alpha/Q}$ derived from $H$, the $P_\alpha[G]$-generic. $C_{\alpha/Q}$ is of the form $C_{\alpha/\beta}$ if $Q=P_\beta$, then $C_{\alpha/Q} \subseteq (\beta,\alpha]$, if $Q=P_\beta * P_{\dot{\gamma}/\beta}$ and $\gamma$ is as above, then $C_{\alpha /Q} \subseteq (\gamma,\alpha]$.
In both cases, $\max(C_{\alpha/Q})=\alpha$.
If $\circ(\alpha)=0$, then $C_{\alpha/Q} \setminus \{\alpha\}$ is bounded in $\alpha$, otherwise, $C_{\alpha/Q} \setminus \{\alpha\}$ adds a club subset of $\alpha$, and singularizes $\alpha$ to have cofinality $\cf(\omega^{\circ(\alpha)})$ (the exponentiation is the ordinal one).
Furthermore, $C_\alpha=C_Q \sqcup C_{\alpha/Q}$.

\end{enumerate}

\end{prop}

Fix $\alpha$.
We assume that for $\beta<\alpha$ and $P_\beta$-name $\dot{\nu}$ with $\Vdash_\beta \beta \leq \dot{\nu}<\alpha$, $P_{\dot{\nu}}/P_\beta$ are defined and satisfy Proposition \ref{theproperties}.
Our main task is to define $P_\alpha$ and $P_\alpha/P_\beta$ (or more precisely, $P_{\dot{\nu}}/P_\beta$ when $\dot{\nu}$ is forced in $P_\beta$ to be $\alpha$) for $\beta \leq \alpha$, and show that all the relevant forcings at level $\alpha$ satisfy Propsition \ref{theproperties}.

\begin{defn}{\textbf{(The forcing: general cases)}}
\label{forcingwithtop}
A condition in $P_\alpha$ is either the weakest condition $1_\alpha$, or of the form $p=s^\frown \langle \alpha,A,F \rangle$, where $s=\langle s_i \rangle_{i<n}$ for some $n$, $s_i$ is either $\langle \alpha_i,A_i,F_i \rangle$ or $\langle \alpha_i,A_i,F_i,\dot{\nu}_i,\dot{q}_i \rangle$ such that

\begin{enumerate}

\item \label{technicalsplit} for $i<n$, if $\circ(\alpha_i) \leq \max_{j<i}(\circ(\alpha_j))$, then $s_i$ is a triple, otherwise $s_i$ is a $5$-tuple.

\item $\alpha_0< \cdots < \alpha_{n-1}<\alpha_n=:\alpha$ are inaccessible.

\item for $i \leq n$,

\begin{itemize}

\item $A_i \in \vec{U}(\alpha_i)$ (note that $A_i$ is empty if and only if $\circ(\alpha_i)=0$).

\item $\dom(F_i)= \{\xi \in A_i \mid \circ(\xi)> \max_{j<i}(\circ(\alpha_j))\}$.

\item for $\gamma \in \dom(F_i)$, $F_i(\gamma)=\langle \dot{\nu},\dot{q} \rangle$, $\Vdash_\gamma \dot{\nu} \in [\gamma,\alpha_i)$, $\dot{q} \in P_{\dot{\nu}}/P_\gamma$.

\item if $i \neq n$, $\dot{\nu}_i$ and $\dot{q}_i$ exist, then $\Vdash_{\alpha_i} \dot{\nu}_i \in [\alpha_i,\alpha_{i+1})$ and $ \dot{q}_i \in P_{\dot{\nu}_i}/P_{\alpha_i}$.

\item if $i \neq n$, recall that $\bar{\nu}_i=\sup\{\nu \mid \exists r \in P_\beta (r \Vdash_\beta \dot{\nu}_i=\check{\nu})\}+1$. 
Note that $|P_{\alpha_i}|\leq \alpha_i^+$, so $\bar{\nu}_i<\alpha_{i+1}$.
We assume that $\min(A_{i+1})>\bar{\nu}_i$.

\end{itemize}

\end{enumerate}

We assume in general that if $\gamma<\gamma^\prime$ are both in $A$ or in $A_i$, and $\dot{\nu}^\gamma=F(\gamma)_0$, then $\bar{\nu}^\gamma<\gamma^\prime$.
We denote $s$ by $\stem(p)$, $\alpha_{n_1}+1$ by $\tp(s)$, and $n$ by $l(p)$.
Define $\circ(s)=\max\{\circ(\alpha_i) \mid i<l(p)\}$ if exists, otherwise $\circ(s)=-1$.
We sometimes refer $\alpha_i$ as $(s_i)_0$.
We say that $p$ is {\em pure} if $p=\langle \alpha,A,F \rangle$, in which case, $\dom(F)=A$, otherwise, we say that $p$ is {\em impure}.
For each condition $p$, $\stem(p) \in P_{\alpha_{n-1}} \cup P_{\alpha_{n-1}} * P_{\dot{\nu}_{n-1}}/P_{\alpha_{n-1}}$ for some $P_{\alpha_{n-1}}$-name of an ordinal $\dot{\nu}_{n-1}$.
Also, note that if $\circ(\alpha_i)\leq \max_{j<i} \circ(\alpha_j)+1$, then $F_i=\emptyset$.
Thus, we can see that if $i<n$ and $s_i$ is a triple, then $F_i=\emptyset$.
The reason that we do not remove $F_i$ from a triple is because we always consider $\langle s_k \rangle_{k<i} {}^\frown \langle \alpha_i,A_i,F_i \rangle$ as a condition in $P_{\alpha_i}$.
One may wonder if it may be simpler by saying that $\dom(F_i)$ is $A_i$, however, by the transitivity of the relations, the restrictions of the domains or ordinals of certain Mitchell orders are necessary.

\end{defn}

\begin{defn}{\textbf{(The direct extension relation)}}
\label{directext}

First, let $p \leq^*_\alpha 1_{P_\alpha}$ for all $p \in P_\alpha$.
Let $p=s^\frown \langle \alpha,A, F \rangle$ and $p^\prime=s^\prime {}^\frown \langle \alpha,A^\prime,F^\prime \rangle$.
Write $s=\langle s_i \rangle_{i<l(p)}$, and $s^\prime=\langle s_i^\prime \rangle_{i<l(p^\prime)}$.
Let $A_{l(p)}=A, F_{l(p)}=F, A_{l(p^\prime)}=A^\prime$, and $F_{l(p^\prime)}=F^\prime$.
We say that
$p$ is a {\em direct extension of $p^\prime$}, denoted by $p \leq^*_\alpha p^\prime$, if

\begin{enumerate}

\item $n:=l(p)=l(p^\prime)$.

\item for $i \leq n$,

\begin{itemize}

\item $\alpha_i=\alpha_i^\prime$.

\item $A_i \subseteq A_i^\prime$.

\item for $\gamma \in A_i$, 

\begin{center}

$s \restriction i^\frown \langle \gamma, A_i \restriction \gamma, F_i  \restriction (\dom(F_i) \cap A_i \restriction \gamma) \rangle \leq^*_\gamma s^\prime \restriction i^\frown \langle \gamma, A_i^\prime \restriction \gamma, F_i^\prime \restriction (\dom(F_i^\prime) \cap A_i \restriction \gamma)\rangle$,

\end{center}

and furthermore, if $\gamma \in \dom(F_i)$,

\begin{center}

$s \restriction i^\frown \langle \gamma,  A_i \restriction \gamma, F_i \restriction (\dom(F_i) \cap A_i \restriction \gamma) \rangle \Vdash_\gamma F_i(\gamma)_0=F_i^\prime(\gamma)_0$ and $F_i(\gamma)_1 \leq^*_{F_i(\gamma)_0/\gamma} F_i^\prime(\gamma)_1$.

\end{center}
 
\item $s \restriction i^\frown \langle \alpha_i,A_i,F_i \rangle \leq^*_{\alpha_i} s^\prime \restriction i^\frown \langle \alpha_i,A_i,F_i \rangle$, and if $s_i$ is a $5$-tuple (note that this is true if and only if $s_i^\prime$ is a $5$-tuple), then 

\begin{center}

$s \restriction i^\frown \langle \alpha_i, A_i, F_i \rangle \Vdash_{\alpha_i} \dot{\nu}_i=\dot{\nu}_i^\prime$ and $\dot{q}_i \leq^*_{\dot{\nu}_i/\alpha_i} \dot{q}_i^\prime$.

\end{center} 
 
\end{itemize}

\end{enumerate}

\end{defn}

We now exemplify how non-direct extension works.
We first start by describing an instance of a minimal one-step extension, as it naturally occurs as a part of an extension relation.

To exemplify, assume $\circ(\alpha)=\omega$ and $p=\langle \alpha,A,F \rangle$ is pure.
Let $\gamma \in A$.
Assume $\circ(\gamma)=5$, $\dot{\nu}=F(\gamma)_0$, and $\bar{\nu}=\sup\{\nu \mid \exists p \in P_\gamma (p \Vdash_\gamma \dot{\nu}=\check{\nu})\}+1$ is the corresponding value.
Then the minimal one-step extension of $p$ by $\gamma$, which will be denoted by $p+\langle \gamma\rangle $, is the condition

\begin{center}

 $p^\prime=\langle \gamma, A\restriction \gamma, F \restriction (A\restriction\gamma) ,\dot{\nu},\dot{q}\rangle^\frown \langle \alpha, A^\prime, F^\prime \rangle$,
 
 \end{center} 
 
 where $A \restriction \gamma=\{\xi<\gamma \mid \xi \in A$ and $\circ(\xi)<5\}$, $\dot{q}=F(\gamma)_1$, $A^\prime=A \setminus (\gamma+1)$ (note that $\min(A \setminus (\gamma+1))>\bar{\nu}$), and $F^\prime=F \restriction \{\xi \in A^\prime \mid \circ(\xi)>5\}$.
A triple $\langle \gamma, A \restriction\gamma, F \restriction (A \restriction\gamma) \rangle$ is now considered as a pure condition in $P_\gamma$, the $5$-tuple $\langle \gamma,A \restriction \gamma, F\restriction (A\restriction \gamma),\dot{\nu},\dot{q} \rangle$ is considered as a condition in $P_\gamma * P_{\dot{\nu}}/P_\gamma$.

A further one-step extension of $p^\prime$ can be done by either adding $\gamma^\prime \in A \restriction \gamma$ to extend the stem to the left as before, or extending by using $\gamma^\prime \in A^\prime$.
Consider the latter case.
Suppose that $\circ(\gamma^\prime) \leq 5$, then $\gamma^\prime \not \in \dom(F^\prime)$. 
The one-step extension of $p^\prime$ by $\gamma^\prime$ will be

\begin{center}

$\langle \gamma, A \restriction \gamma, F \restriction (A \restriction \gamma), \dot{\nu}, \dot{q} \rangle^\frown \langle \gamma^\prime, A^\prime \restriction \gamma^\prime, F^\prime \restriction (A^\prime \restriction \gamma^\prime) \rangle^\frown \langle \alpha, A^\prime \setminus (\gamma^\prime+1), F^\prime \setminus (\gamma^\prime+1) \rangle$.

\end{center}

We see that $\langle \gamma, A\restriction \gamma, F \restriction (A \restriction \gamma) \rangle^\frown \langle \gamma^\prime, A^\prime \restriction \gamma^\prime, F^\prime \restriction (A^\prime \restriction \gamma^\prime) \rangle \in P_{\gamma^\prime}$.
Note that $F^\prime \restriction (A^\prime \restriction \gamma^\prime)$ is empty since ordinals in $\dom(F^\prime)$ have Mitchell orders at least $6$, but the ordinals in $A^\prime \restriction \gamma^\prime$ have Mitchell order less than $\circ(\gamma^\prime)$, which is at most $5$.
If, however, $\circ(\gamma^\prime)>5$, the one-step extension of $p^\prime$ by $\gamma^\prime$ will be

\begin{center}

$\langle \gamma, A \restriction \gamma, F \restriction(A \restriction\gamma), \dot{\nu}, \dot{q} \rangle^\frown \langle \gamma^\prime, A^\prime \restriction \gamma^\prime, F^\prime \restriction (\dom(F^\prime) \cap A^\prime \restriction \gamma^\prime) ,F^\prime(\gamma^\prime)_0, F(\gamma^\prime)_1 \rangle^\frown \langle \alpha, A^{\prime\prime} ,F^{\prime\prime} \rangle$,

\end{center}

for some certain $A^{\prime\prime}$ and $F^{\prime\prime}$.
Suppose that $\circ(\gamma^\prime)=8$ and $\dot{\nu}^\prime=F^\prime(\gamma^\prime)_0$.
Then, ordinals in $\dom(F^\prime \restriction (A^\prime \restriction \gamma^\prime))$ have Mitchell orders $6$ or $7$.
$A^{\prime \prime}=A^\prime \setminus (\gamma^\prime+1)$ (recall $\min(A^\prime \setminus (\gamma^\prime+1))>\bar{\nu}^\prime$).
$F^{\prime \prime}=F^\prime \restriction \{\xi \in A^{\prime \prime} \mid \circ(\xi)>8\}$.
We now formally define a minimal one-step extension.

\begin{defn}{\textbf{(A minimal one-step extension)}}
\label{minimalonestep}
Let $p \in P_\alpha$, $n=l(p)$, $p=s^\frown \langle \alpha, A,F \rangle$, for $i<n$, write $s_i=\langle \alpha_i,A_i,F_i \rangle$ or $\langle \alpha_i,A_i,F_i,\dot{\nu}_i,\dot{q}_i \rangle$.
Let $A_n=A$, $F_n=F$, $i \leq n$, and $\gamma \in A_i$.
The {\em one-step extension of $p$ by $\gamma$}, denoted by $p+\langle \gamma \rangle$, is the condition $p^\prime=s^\prime {}^\frown \langle \alpha,A^\prime,F^\prime \rangle$ such that

\begin{enumerate}

\item $l(p^\prime)=l(p)+1$.

\item $s^\prime \restriction i=s \restriction i$.

\item if $\gamma \not \in \dom(F_i)$ (i.e. $\circ(\gamma) \leq \max_{j<i} \circ(\alpha_j))$), $s^\prime_i=\langle \gamma, A_i \restriction \gamma,F_i \restriction (\dom(F_i) \cap A_i \restriction \gamma)\rangle$, otherwise, $s^\prime_i=\langle \gamma,A_i \restriction \gamma, F_i \restriction (\dom(F_i) \cap A_i \restriction \gamma), F_i(\gamma)_0,F_i(\gamma)_1 \rangle$.

\item if $i<n$, then,

\begin{itemize}

\item if $s_{i+1}$ is a triple, then $s_{i+1}^\prime=\langle \alpha_i, A_i \setminus (\gamma+1), F_i \restriction \{ \xi \in A_i \setminus (\gamma+1) \mid \circ(\xi)> \circ(\gamma)\} \rangle$ (note that if $\gamma \not \in (\dom(F_i))$, then $F_i \restriction \{ \xi \in A_i \setminus (\gamma+1) \mid \circ(\xi)> \circ(\gamma)\}$ is simply $F_i \restriction (A_i \setminus (\gamma+1))$).

\item if $s_{i+1}$ is a $5$-tuple, then  $s_{i+1}^\prime=\langle \alpha_i, A_i \setminus (\gamma+1), F_i \restriction \{ \xi \in A_i \setminus (\gamma+1)\mid \circ(\xi)> \circ(\gamma)\},\dot{\nu}_i,\dot{q}_i \rangle$ (note that if $\gamma \not \in (\dom(F_i))$, then $F_i \restriction \{ \xi \in A_i \setminus (\gamma+1) \mid \circ(\xi)> \circ(\gamma)\}$ is simply $F_i \restriction (A_i \setminus (\gamma+1))$).

\item for $j>i+1$, $s^\prime_j=s_j$, $A^\prime=A$, $F^\prime=F$.

\end{itemize}

\item if $i=n$, then $A^\prime=A \setminus (\gamma+1), F^\prime=F \restriction \{ \xi \in A_i \setminus (\gamma+1) \mid \circ(\xi)> \circ(\gamma)\}$ (note that if $\gamma \not \in \dom(F)$, then $F \restriction \{ \xi \in A_i \setminus (\gamma+1) \mid \circ(\xi)> \circ(\gamma)\}$ is simply $F\restriction (A \setminus (\gamma+1))$).

\end{enumerate}

\end{defn}

We define an $n$-step extension recursively by $p+\langle \vec{\gamma} \rangle=(p+\langle \gamma_0,\cdots, \gamma_{n-2} \rangle)+\langle \gamma_{n-1}\rangle$, as long as the extensions are legitimate.

\begin{defn}{\textbf{(The extension relation)}}
\label{extrel}
With the setting in Definition \ref{directext}, we say that $p$ is an {\em extension} of $p^\prime$, denoted by $p \leq_\alpha p^\prime$, if

\begin{enumerate}

\item $l(p) \geq l(p^\prime)$.

\item \label{extpivot} there are $i_0<i_1< \cdots <i_{l(p^\prime)-1}<i_{l(p^\prime)}=l(p)$ such that

\begin{itemize}

\item $\alpha_{i_j}=\alpha_j^\prime$.

\item $A_{i_j} \subseteq A_j^\prime$.

\item for $\gamma \in A_{i_j}$,

\begin{center}

$s \restriction i_j^\frown \langle \gamma, A_{i_j} \restriction \gamma, F_{i_j} \restriction (\dom(F_{i_j} \cap A_{i_j} \restriction \gamma) \rangle \leq_\gamma s^\prime \restriction j^\frown \langle \gamma_,A_j^\prime \restriction \gamma, F_j^\prime \restriction (\dom(F_j^\prime) \cap A_j^\prime \restriction \gamma) \rangle$,

\end{center}

and if $\gamma \in \dom(F_{i_j})$,

\begin{center}

$s \restriction i_j^\frown \langle \gamma^\frown A_{i_j} \restriction \gamma, F_{i_j} \restriction  (\dom(F_{i_j} \cap A_{i_j} \restriction \gamma) \rangle \Vdash_\gamma F_{i_j}(\gamma)_0=F_j^\prime(\gamma)_0$ and $F_{i_j}(\gamma)_1 \leq^*_{F_{i_j}(\gamma)_0/\gamma} F_j^\prime(\gamma)_1$ (the $\leq^*$ relation is intended).

\end{center}

\item $s \restriction i_j^\frown \langle \alpha_{i_j},A_{i_j},F_{i_j} \rangle \leq_{\alpha_{i_j}} s^\prime \restriction j^\frown \langle \alpha_j,A_j,F_j \rangle$, and $s_{i_j}$ is a $5$-tuple, iff $s_j^\prime$, in which case

\begin{center}

$s \restriction i_j^\frown \langle \alpha_{i_j}, A_{i_j}, F_{i_j} \rangle \Vdash_{\alpha_{i_j}} \dot{\nu}_{i_j}=\dot{\nu}_j^\prime$ and $\dot{q}_{i_j} \leq_{\dot{\nu}_{i_j}/\alpha_{i_j}} \dot{q}_j^\prime$.

\end{center}

\end{itemize}

\item let $i_0,i_1, \cdots, i_{l(p^\prime)}$ be as in (\ref{extpivot}),
 $k \not \in \{i_0, \cdots, i_{l(p^\prime)}\}$,
and $j$ be the least such that $k<i_j$.
Then,

\begin{itemize}

\item $\alpha_k \in A_j^\prime$.

\item $A_k \subseteq A_j^\prime \restriction \alpha_k$.

\item for $\gamma \in A_k$,

\begin{center}

$s \restriction k^\frown \langle \gamma, A_k \restriction \gamma, F_k \restriction (\dom(F_k) \cap A_k \restriction \gamma) \rangle \leq_\gamma s^\prime \restriction j-1^\frown \langle \gamma,A_j^\prime \restriction \gamma, F_j^\prime \restriction (\dom(F_j^\prime) \cap A_j^\prime \restriction \gamma) \rangle$, 

\end{center}

and if $\gamma \in \dom(F_k)$,

\begin{center}

$s \restriction k^\frown \langle \gamma, A_k \restriction \gamma, F_k \restriction (\dom(F_k) \cap A_k \restriction \gamma) \rangle \Vdash_\gamma F_k(\gamma)_0=F_j^\prime(\gamma)_0$ and $F_k(\gamma)_1 \leq^*_{F_k(\gamma)_0/\gamma} F_j^\prime(\gamma)_1$ (the $\leq^*$ relation is intended).

\end{center}

\item  $\alpha_k \in \dom(F_j^\prime)$ if and only if $s_k$ is a $5$-tuple, in which case, $\Vdash_{\alpha_k} \dot{\nu}_k=F_j^\prime(\alpha_k)_0$, and 

\begin{center}

$s \restriction k^\frown \langle \alpha_k, A_k,F_k \rangle \Vdash_{\alpha_k} \dot{q}_k \leq_{\dot{\nu}_k/\alpha_k} F_j^\prime(\alpha_k)_1$.

\end{center}

\end{itemize}

\end{enumerate}

\end{defn}

It is not true that $p \leq_\alpha p^\prime$ iff $p \leq^*_\alpha p^\prime+ \vec{\gamma}$ for some $\vec{\gamma}$.
However, $p \leq_\alpha p^\prime$ iff there is $p^* \leq^*_\alpha p^\prime+\vec{\gamma}$ for some $\vec{\gamma}$ such that $p$ is obtained from $p^*$ by extending the names appearing at the fifth-coordinate of each $5$-tuple sequence in $\stem(p^*)$.
We now consider some basic properties of $P_\alpha$.
\begin{prop}
\label{cc}
$|P_\alpha| \leq \alpha^+$, and hence $\alpha^{++}$-c.c.

\end{prop}

\begin{proof}

Each condition $p=s^\frown \langle \alpha,A, F \rangle$, we have $s \in V_\alpha$, $A \subseteq \alpha$, and $F$ is a partial function from $\alpha$ to $V_\alpha$.
Since GCH holds and $\alpha$ is inaccessible, $|P_\alpha| \leq 2^\alpha=\alpha^+$.

\end{proof}

We now prove the strong form of closure under the direct extension relation, i.e. Proposition \ref{theprikryproperty} (\ref{nameclosed}).

\begin{prop}
\label{strongclosure}
Let $p \in P_\alpha$ and $(\stem(p)_0)_0>\gamma$.
Let $\xi< \min(X \setminus (\gamma+1))$.
Assume for $\epsilon<\xi$, there is $p_\epsilon$ such that $p_\epsilon \leq^*_\alpha p$.
Then there is $p^* \leq^*_\alpha p$ such that for $\epsilon$, $p^* \leq^*_\alpha p_\epsilon$.

\end{prop}

\begin{proof}

Write $p_\epsilon=\stem(p_\epsilon)^\frown \langle \alpha, A^\epsilon,F^\epsilon \rangle$.
Let $Q$ be the forcing such that $\stem(p) \in Q$.
Then for all $\epsilon$, $\stem(p_\epsilon) \in Q$.
By proposition \ref{theproperties} (\ref{nameclosed}), and the fact that $(\stem(p)_0)_0>\gamma$, let $s^*$ be such that for all $\epsilon$, $s^* \leq^* \stem(p_\epsilon)$.
Take $A^*=\cap_\epsilon A^\epsilon$ and for $\beta \in A^*$, we will take $F^*(\beta)_1$ as a name with is forced to be $\leq^*$-stronger than $F^\epsilon(\beta)_1$ for all $\epsilon$.
More explicitly, we do this by induction on ordinals in $A^*$.
First, let $\beta=\min(A^*)$.
Consider $t_\epsilon= \stem(p_\epsilon)^\frown \langle \beta, A^\epsilon \restriction \beta, F^\epsilon \restriction \{\xi \in \dom(F^\epsilon) \cap  A^\epsilon \restriction \beta\mid \circ(\xi)>\circ(\beta) \}$.
For each $G$ which is $P_\beta$-generic, let $F^*(\beta)_1[G]$ be such that for any $\epsilon$ with $t_\epsilon \in G$, $F^*(\beta)_1[G] \leq^* F^\epsilon(\beta)_1[G]$.
Inductive steps can be done similarly.
Now, let $F^*(\beta)=\langle F^p(\beta)_0,F^*(\beta)_1 \rangle$.
Let $p^*=s^* {}^\frown \langle \alpha,A^*,F^* \rangle$, then $p^*$ is as required.

\end{proof}

In particular, we have that

\begin{coll}
\label{closure}
$(\{p \in P_\alpha \mid (\stem(p)_0)_0>\gamma\} ,\leq^*_\alpha)$ is $\min(X \setminus (\gamma+1))$-closed.
In particular, $(\{ p \in P_\alpha \mid p$ is pure$\},\leq^*_\alpha)$ is $\alpha$-closed.

\end{coll}

\begin{rmk}
When we restrict $P_\alpha$ to the set $\{ p \in P_\alpha \mid p$ is pure$\}$, $\leq^*_\alpha$ and $\leq_\alpha$ coincide.
In addition, if $\circ(\alpha)=0$, a pure condition in $P_\alpha$ is simply a trivial forcing, which makes $(\{ p \in P_\alpha \mid p$ is pure$\},\leq^*_\alpha)$ $\infty$-closed.
\end{rmk}

\begin{defn}{\textbf{(The quotient forcing)}}
\label{quotient}
Let $\beta<\alpha$, and $\dot{\nu}$ is a $P_\beta$-name of an ordinal in the interval $[\beta,\alpha)$.
Let $Q$ be either $P_\beta$ or $P_\beta * P_{\dot{\nu}}/P_\beta$-generic.
Let $G$ be $Q$-generic.
We define $P_\alpha[G]$ as the collection of $s_1[G]^\frown \langle \alpha, A, F[G] \rangle$ where there is $s_0 \in G$, $s_0^\frown s_1{}^\frown \langle \alpha,A,F \rangle \in P_\alpha$, and $F[G](\gamma)=F(\gamma)[G]$ for all $\gamma \in \dom(F)$.

We note that those notions make sense due to Proposition \ref{theproperties} (\ref{factor1}), (\ref{factor2}), and (\ref{twolevelquotient}) below $\alpha$, and the fact that $\min(A)>\bar{\nu}>\beta$.
We say that $P_\alpha/Q$ is the collection of $Q$-name $\dot{q}$ such that $\Vdash_\beta \dot{q} \in P_\alpha[\dot{G}]$ where $\dot{G}$ is the canonical name for a generic for $Q$.
We assume that $\dot{q}$ is also a nice name.
The extension relation, direct extension relation, minimal one-step extension are those induced from $P_\alpha$.
Note that it is possible that $P_\alpha[G]$ has no conditions of the form $\langle \alpha,A,F \rangle$ (for example, item $(2)$ in Definition \ref{basicsecondcase}).
In that case, we do not define pure conditions, impure conditions, and stems.
Assume $P_\alpha[G]$ has conditions of the form $\langle \alpha,A,F \rangle$. $P_\alpha[G]$, a {\em pure} condition is of the form $\langle \alpha,A,F \rangle$.
Other conditions which are not $1_{P_\alpha[G]}$ are called {\em impure}.
We will call the corresponding relations $\leq_{\alpha/Q}$ and  $\leq^*_{\alpha/Q}$, respectively.
If $Q=P_\beta$, we abbreviate $\leq_{\alpha/Q}$ by $\leq_{\alpha/\beta}$, and so on.

\end{defn}

\begin{defn}

Define $P_\alpha/P_\alpha$ as the trivial forcing (in $V^{P_\alpha}$).

\end{defn}

\begin{prop}

Let $\beta$, $\dot{\nu}$, $Q$, and $G$ be as in Definition \ref{quotient}.
Then, 

\begin{enumerate}

\item $\Vdash_Q |P_\alpha/Q| \leq \alpha^+$, so $P_\alpha/Q$ is $\alpha^{++}$-c.c. Furthermore, $|P_\alpha/Q| \leq \alpha^+$.

\item Suppose that pure conditions in $P_\alpha/Q$ can be defined, $\Vdash_Q \dot{q} \in P_\alpha/Q$ is pure.
Fix $\xi<\alpha$ and for $\epsilon<\xi$, there are $p_\epsilon \in Q$ and $\dot{q}_\epsilon$ such that $p_\epsilon \Vdash_Q \dot{q}_\epsilon \leq^* \dot{q}$.
Then there is $\dot{q}^*$ such that $\Vdash_Q \dot{q}^* \leq^* \dot{q}$, and for all $\epsilon$, $p_\epsilon \Vdash_Q \dot{q}^* \leq^* \dot{q}_\epsilon$.

\item Suppose that pure conditions in $P_\alpha[G]$ are defined, then $ \{p \in P_\alpha[G] \mid p$ is pure$\}, \leq^*)$ is $\alpha$-closed.

\item Let $\gamma< \min(X \setminus (\beta+1))$.
Suppose that $\Vdash_\beta \dot{q} \in P_\alpha/P_\beta$.
For $\epsilon<\gamma$, let $p_\epsilon\in P_\beta$ and $\dot{q}_\epsilon$ be such that $p_\epsilon \Vdash_\beta \dot{q}_\epsilon \leq^* \dot{q}$.
Then there is $\dot{q}^*$ such that $\Vdash_\beta \dot{q}^* \leq^* \dot{q}$ and for all $\epsilon$, $\dot{q}^* \leq^* \dot{q}_\epsilon$.
 In particular, $\Vdash_Q  (P_\alpha/P_\beta,\leq^*)$ is $\min(X \setminus (\beta+1))$-closed.

\end{enumerate}

\end{prop}

\begin{proof}

\begin{enumerate}

\item By Proposition \ref{theproperties} (\ref{prikry}) and (\ref{ccandclosed}), $Q$ preserves cardinals and GCH, by the same counting argument as in Proposition \ref{cc}, $|P_\alpha[G]|\leq \alpha^+$ in $V[G]$.
Now since $|Q| \leq \alpha^+$ and $\Vdash_Q |P_\alpha/Q| \leq \alpha^+$ and GCH holds in $V$, we have $|P_\alpha/Q|\leq \alpha^+$.

\item Let $G$ be $Q$-generic.
In $V[G]$, write $\dot{q}[G]=\langle \alpha,A,F \rangle$ and for $\epsilon$ with $p_\epsilon \in G$, set $\dot{q}_\epsilon[G]=\langle \alpha, A^\epsilon,F^\epsilon \rangle$.
Take $A^*=\cap_{\{\epsilon \mid p_\epsilon \in G\}} A^\epsilon$ and for $\beta \in A^*$.
Set $F^*(\gamma)_1$ such that for each $\epsilon$ $F^*(\gamma)_1$ is forced to be $\leq^*$-stronger than $F^\epsilon(\gamma)_1$, and let $F^*(\gamma)=\langle F^p(\gamma)_0, F^*(\gamma)_1 \rangle$.
Let $q^*=\langle \alpha,A^*,F^* \rangle$ and $\dot{q}^*$ be a name of such $q^*$.
Then we can check that $\dot{q}^*$ is as required.

\item It is a consequence of (2).

\item The proof is similar to $(2)$.

\end{enumerate}

\end{proof}

We now discuss on factorizations on $P_\alpha$.
Let $\beta,\dot{\nu}$, $Q$, and $G$ be as in Definition \ref{quotient}.
Let $H$ be $P_\alpha[G]$-generic.
Then, let $I$ be the collection of $s_0^\frown s_1^\frown \langle \alpha,A,F \rangle \in P_\alpha$ such that $s_0 \in G$ and $(s_1^\frown \langle \alpha,A, F \rangle)[G] \in H$.
Then $I$ is $P_\alpha$-generic, and $V[I]=V[G][H]$.
On the contrary, let $I$ be $P_\alpha$-generic.
Let $p \in I$ (say $p \neq 1_{P_\alpha})$.
Let $p=s_0^\frown s_1^\frown \langle \alpha,A,F \rangle$.
Set $Q^\prime$ be the forcing containing $s_0$.
Let $G=\{t_0 \in Q^\prime \mid$ there are $t_1,A^\prime,$ and $F^\prime$ such that $t_0^\frown t_1^\frown \langle \alpha,A^\prime,F^\prime \rangle \in I\}$.
Then, $G$ is $Q^\prime$-generic.
Let $H=\{(t_1^\frown \langle\alpha,A^\prime,F^\prime \rangle)[G] \mid \exists t_0 \in G (t_0^\frown t_1^\frown \langle \alpha,A^\prime,F^\prime \rangle) \in I\}$.
Then, $H$ is $P_\alpha/G$-generic over $V[G]$, and $V[I]=V[G][H]$.
Hence, Proposition \ref{theproperties} (\ref{factor1}) and (\ref{factor2}) follow at level $\alpha$.

We further discuss on $P_\alpha[G]/P_\gamma[G]$.
Again, let $\beta,\dot{\nu},Q,$ and $G$ be as in Definition \ref{quotient}.
We define $P_\alpha[G]/P_\gamma[G]$ as follows.
Recall that by Proposition \ref{theproperties} (\ref{factor1}), if $H$ is $P_\gamma[G]$-generic, then there is the canonical $P_\alpha$-generic $I$ with $V[G][H]=V[I]$.
Define $P_\alpha[G]/P_\gamma[G]$ as the collection of $P_\gamma[G]$-name $\dot{q}$ such that $\Vdash_{P_\gamma[G]} \dot{q} \in P_\alpha[G][\dot{I}]$, where $\dot{H}$ is the name for $P_\gamma[G]$-generic, and $\dot{I}$ is the name for $P_\gamma$-canonical generic such that $\Vdash_{P_{\gamma}[G]} V[G][\dot{H}]=V[\dot{I}]$.
Finally, for $\dot{q}$ with $\Vdash_\gamma \dot{q} \in P_\alpha/P_\gamma$, define $\dot{q}[G]$ as a $P_\gamma[G]$-name of an element in $P_\alpha[G]/P_\gamma[G]$ as follows.
For each $H$ which is $P_\gamma[G]$-generic, and $I=G*H$, we see that $\dot{q}[I] \in P_\alpha[I]$.
Find $p \in P_\alpha$ such that $p=s_0^\frown s_1$ where $s_0 \in I$.
Let $\dot{q}[G][H]=s_1[I]$.
We need that $\dot{q}[G][H]=s_1[I]$ for each such $H$ and the corresponding generic $I$.
One can trace these genericities to show that Proposition \ref{theproperties} (\ref{twolevelquotient}) holds at level $\alpha$.

\section{The Prikry property}
\label{theprikryproperty}
In this section we give a proof of the Prikry property at level $\alpha$.
If $\circ(\alpha)=0$, then $P_\alpha$ is equivalent to $P_\beta$ or $P_\beta * P_{\dot{\nu}}/P_\beta$ for some $\beta,\dot{\nu}$ such that $\Vdash_\beta \dot{\nu} \in [\beta,\alpha)$.
These are either a Prikry-type forcing or a two-step iteration of Prikry-type forcings, hence, has the Prikry property.
We assume throughout this section that $\circ(\alpha)>0$.
Note that when $p=\langle \alpha,A,F\rangle$ is pure, then $\dom(F)=A$.
However,  for each impure condition $p=s^\frown \langle \alpha,A,F \rangle$, it is possible that $\dom(F) \neq A$.
Our proofs in this section focus on the pure conditions, with some modification, they also work for impure conditions.

\begin{lemma}
\label{maximaltop}
Let $D \subseteq P_\alpha$ be a dense open set, $p=\langle \alpha,A,F \rangle \in P_\alpha$.
Then there is $p^\prime=\langle \alpha, A^\prime, F^\prime \rangle$ such that for each stem $s$, if $s^\frown \langle \alpha,B,F^\prime \rangle \leq_\alpha p^\prime$ is in $D$, then $s^\frown \langle \alpha, A^\prime \setminus \tp(s), F^\prime \restriction \{ \gamma \in A^\prime \setminus \tp(s) \mid \circ(\gamma)>\circ(s)\} \rangle \in D$.

\end{lemma}

\begin{proof}

We build $\langle A_\gamma, F_\gamma \mid \gamma \in A \rangle$ as follows:

\begin{enumerate}

\item $\langle A_\gamma \mid \gamma \in A \rangle$ is $\subseteq$-decreasing, $A_\gamma \subseteq A$, $A_\gamma \in \vec{U}(\alpha)$, and $\min(A_\gamma)>\gamma$.

\item for $\gamma^\prime<\gamma<\xi$ and $\xi \in A_{\gamma}$, we have that 

\begin{itemize}

\item $\langle \alpha, A_{\gamma} \restriction \xi, F_{\gamma} \restriction (A_\gamma \restriction \xi) \rangle \leq^*_\xi \langle \alpha, A_{\gamma^\prime} \restriction \xi, F_{\gamma^\prime} \restriction (A_{\gamma^\prime} \restriction \xi) \rangle$.

\item $\langle \alpha, A_\gamma \restriction \xi, F_\gamma \restriction (A_\gamma \restriction \xi) \rangle \Vdash_\xi F_\gamma(\xi)_0=F_{\gamma^\prime}(\xi)_0$ and $F_{\gamma^\prime}(\xi)_1 \leq^* F_{\gamma^\prime}(\xi)_1$ (we assume that the name $F_\gamma(\xi)_0$ is always the same, which is $F(\xi)_0$).

\end{itemize}

\item \label{maxconstruction}  let $\dot{\nu}_\gamma=F(\gamma)_0$.
For each $s \in P_\gamma \cup P_\gamma * P_{\dot{\nu}_\gamma}/P_\gamma$, if $s^\frown \langle \alpha,B,H \rangle \leq^*_\alpha s^\frown \langle \alpha, A_\gamma, F_\gamma \rangle$ is in $D$, then already $s^\frown \langle \alpha, A_\gamma \setminus \tp(s), F_\gamma\restriction \{ \gamma \in A^\gamma \setminus \tp(s) \mid \circ(\gamma)>\circ(s)\}\rangle \in D$.

\end{enumerate}

Suppose for $\gamma^\prime<\gamma$, $A_{\gamma^\prime}$ and $F_{\gamma^\prime}$ is built (note that the case $\gamma=\min(A)$ can be done in a slightly simpler setting).
We now build $A_\gamma$ and $F_\gamma$.
Let $A_\gamma^\prime=\cap_{\gamma^\prime<\gamma} A_{\gamma^\prime} \setminus(\gamma+1)$.
For $\xi \in A_\gamma^\prime$, let $F_\gamma^\prime(\xi)_0=F(\xi)_0$ and $F_\gamma^\prime(\xi)_1$ be a name such that $F_\gamma^\prime(\xi)_1$ is forced to be $\leq^*$-stronger than $ F_{\gamma^\prime}(\xi)_1$ for all $\gamma^\prime<\gamma$.
This is possible since by Proposition \ref{theproperties}, item (\ref{nameclosed}), $\Vdash_\xi (P_{F(\xi)_0}/P_\xi, \leq^*)$ is $\xi^+$-closed, and $\xi>\gamma$.
Let $p^\prime_\gamma=\langle \alpha, A_\gamma^\prime,F_\gamma^\prime \rangle$.
Let $\dot{\nu}_\gamma=F(\gamma)_0$ and $\bar{\nu}_\gamma$ be the corresponding value.
By Proposition \ref{theproperties}, item (\ref{ccandclosed}), let $\langle s_\delta \mid \delta < \bar{\nu}_\gamma^+ \rangle$ be an enumeration of conditions in $P_\gamma \cup P_\gamma * P_{\dot{\nu}}/P_\gamma$.
For each $\delta$, if there is $s_{\gamma,\delta} {}^\frown \langle \alpha,B,H \rangle \leq^*_\alpha s_{\gamma,\delta} {}^\frown p_\gamma^\prime$ being in $D$, we record $A_{\gamma,\delta}=B$, and for $\xi \in \dom(H)$, let $\dot{q}_{\gamma,\delta}=H(\xi)_1$.
Otherwise, let $A_{\gamma,\delta}=A_\gamma^\prime$ and for $\xi \in \dom(H)$, let $\dot{q}_{\gamma,\delta}=F_\gamma^\prime(\xi)_1$.
We now let $A_\gamma= \cap_{\delta<\bar{\nu}_\gamma^+} A_{\gamma,\delta}$, and for $\xi \in A_\gamma$, let $F_\gamma(\xi)_1$ be a name such that $F_\gamma(\xi)$ is forced to be $\leq^*$-stronger than $\dot{q}_{\gamma,\delta}$ for all $\delta$.
Again, this is possible by Proposition \ref{theproperties}  (\ref{nameclosed}).

We now let $A^\prime=\triangle_{\gamma \in A} A_\gamma$.
For each $\gamma \in A^\prime$, let $F^\prime(\gamma)_1$ be a name which is forced to be $\leq^*$-stronger than $F_{\gamma^\prime}(\gamma)_1$ for all $\gamma^\prime<\gamma$.
Again, this is possible by Proposition \ref{theproperties} (\ref{nameclosed}).
Let $F^\prime(\gamma)=\langle F(\gamma)_0, F^\prime(\gamma)_1 \rangle$ and $p^\prime=\langle \alpha,A^\prime,F^\prime \rangle$.
It remains to show that $p^\prime$ satisfies our criteria.
Let $s^\frown \langle \alpha,B,H \rangle \leq_\alpha p^\prime$ being in $D$.
Let $Q$ be the forcing such that  $s \in Q$.
Hence, $s=s_{\gamma, \delta}$ for some $\delta$.
Since $B \subseteq A^\prime \setminus (\gamma+1)$, we have that $B \subseteq A_\gamma$.
For each $\xi \in B$, it is forced that $H(\xi)_1 \leq^*F_\gamma(\xi)_1$. 
By (\ref{maxconstruction}), $s^\frown \langle \alpha, A_\gamma \setminus \tp(s),F  \restriction \{ \gamma \in A^\gamma \setminus \tp(s) \mid \circ(\gamma)>\circ(s)\} \rangle \in D$.
Since $D$ is open, $s^\frown \langle \alpha, A^\prime \setminus \tp(s), F^\prime \restriction \{ \gamma \in A^\prime \setminus \tp(s) \mid \circ(\gamma)>\circ(s)\} \rangle \in D$ as required.

\end{proof}

\begin{lemma}
\label{pigeonhole}
Fix $n>0$.
Let $p=\langle \alpha,A,F \rangle \in P_\alpha$.
For each $\vec{\gamma} \in [A]^n$, let $s_{\vec{\gamma}} \leq^*\stem(p+\vec{\gamma})$.
Then there is $p^*=\langle \alpha, A^*, F^* \rangle \leq^*_\alpha p$ such that for every $\vec{\gamma} \in [A^*]^n$, $\stem(p^*+\vec{\gamma}) \leq^* s_{\vec{\gamma}}$.

\end{lemma}

\begin{proof}

We induct on $n$. 
For $n=1$, each $s_{\langle \gamma \rangle}$ is of the form $\langle \gamma,A_\gamma, F_\gamma, \dot{\nu}_\gamma, \dot{q}_\gamma \rangle$.
Use Lemma \ref{integrate} to obtain $A^\prime$ such that for all $\gamma \in A^\prime$, $A^\prime \restriction \gamma \subseteq A_\gamma$.
For $\gamma \in A^\prime$, let $A^\gamma \in \vec{U}(\alpha)$, $A^\gamma \subseteq A^\prime$, and $\dot{q}_\gamma^*$ be such that for $\xi \in A^\gamma \setminus (\gamma+1)$, $F_\xi(\gamma)_1=\dot{q}_\gamma^*$ (this is possible because as in Proposition \ref{theproperties} (\ref{closure}), the number of nice names is small).
Take $A^*=\triangle_\gamma A^\gamma$.
For $\gamma \in A^*$, take $F^*(\gamma)_1$ be a name which is forced to be $\leq^*$-stronger than $\dot{q}_\gamma^*$ and $\dot{q}_\gamma$.
Take $p^*=\langle \alpha, A^*, F^* \rangle$.
For each $\xi \in A^*$, let $t_{\langle \xi \rangle}=\stem(p^*+\langle \xi \rangle)$.
Note that $t_{\langle \xi \rangle}=\langle \xi,A^* \restriction \xi, F^* \restriction (A^* \restriction \xi), F^*(\xi)_0,F^*(\xi)_1 \rangle$.
Since $A^* \restriction \xi \subseteq A_\xi$, and for $\gamma<\xi$, $\xi \in A^\gamma$, we have that $F^*_\xi(\gamma)_1$ is forced to be $\leq^*$-stronger than $\dot{q}_\gamma^*=F_\xi(\gamma)_1$.
Finally, $F^*(\xi)_1$ is forced to be $\leq^*$-stronger than $\dot{q}_\xi$.
Hence, $t_{\langle \gamma \rangle} \leq^* s_{\langle \gamma \rangle}$.

We now assume $n=m+1$ for some $m>0$. 
For each $\vec{\gamma}^\frown \langle \xi \rangle \in [A]^n$, let

\begin{align*}
s_{\vec{\gamma}^\frown \langle \xi \rangle}^0 &=t_{\vec{\gamma}^\frown \langle \xi \rangle} \restriction m, \\
s_{\vec{\gamma}^\frown \langle \xi \rangle}^1 &= (t_{\vec{\gamma}^\frown \langle \xi \rangle})_m.
\end{align*}

Note that for $\vec{\gamma}$, there is $\tau$ such that for each $\xi$, $s_{\vec{\gamma}^\frown \langle \xi \rangle}^0 \in V_{\tau}$ for some $\tau<\alpha$.
Hence, for $\beta<\circ(\alpha)$ and $\vec{\gamma} \in [A]^m$, there are $B_{\vec{\gamma}}(\beta)\in U(\alpha,\beta)$ and $s_{\langle \gamma_0, \cdots, \gamma_{m-1} \rangle}(\beta)$ such that for $\xi \in B_{\vec{\gamma}}(\beta)$, $s_{\langle \gamma_0, \cdots, \gamma_{m-1},\xi \rangle}^0=s_{\langle \gamma_0, \cdots, \gamma_{m-1} \rangle}(\beta)$.
Take $A^\prime= \triangle_{\vec{\gamma}}\cap_\beta B_{\vec{\gamma}}(\beta)$ and let $s_{\langle \gamma_0, \cdots, \gamma_{m-1} \rangle}$ be such that  for all $\beta<\circ(\alpha)$, $s_{\langle \gamma_0, \cdots, \gamma_{m-1} \rangle} \leq^* s_{\langle \gamma_0, \cdots, \gamma_{m-1} \rangle}(\beta)$ and intersect all measure-one sets with $A^\prime$.
Set $p^\prime= \langle \alpha,A^\prime, F \rangle$.
We see that for each $\vec{\gamma}=\langle \gamma_0, \cdots, \gamma_{m-1} \rangle$, $s_{\vec{\gamma}}$ is $\leq^*$-stronger than $\stem(p^\prime+ \vec{\gamma})$.
Apply our induction hypothesis with $p^\prime$ and $s_{\gamma_0,\cdots, \gamma_{m-1}}$ for all $\vec{\gamma}$ to get $p^{\prime \prime} \leq^*_\alpha p^\prime$ such that $\stem(p^{\prime \prime} +\vec{\gamma}) \leq^* s_{\vec{\gamma}}$ for $\vec{\gamma}$ of length $m$.
Let $A^*=A^{\prime \prime}$ and for $\xi \in A^*$, let $F^*(\xi)_1$ be a name which is forced to be $\leq^*$-stronger than $F^{\prime \prime}(\xi)_1$ and $s_{\vec{\gamma}^\frown \langle \xi \rangle}^1$ for all  $\vec{\gamma}$ in $\{\vec{\gamma} \mid |\vec{\gamma}|=m$ and $\max(\vec{\gamma})<\xi\}$.
This is possible since the number of such $\vec{\gamma}$ is $\xi$ and by Proposition \ref{theproperties} (\ref{nameclosed}).
Let $p^*=\langle \alpha,A^*, F^* \rangle$.
One can then check that $p^*$ is as required.

\end{proof}

\begin{prop}
\label{mainprikry}
$P_\alpha$ has the Prikry property.

\end{prop}

\begin{proof}

We first deal with the case $p$ is pure.
The impure case will be done similarly, with just a small additional work in the beginning.

Let $p=\langle \alpha,A,F \rangle$.
Assume that $\min(A)>\circ(\alpha)$.
By Lemma \ref{maximaltop}, we may extend $p$ so that if $s$ is a stem and $s^\frown \langle \alpha,B,H \rangle \leq p$ decides $\sigma$, then $s^\frown \langle \alpha, A \setminus \tp(s), F \restriction \{ \gamma \in A \setminus \tp(s) \mid \circ(\gamma)>\circ(s)\} \rangle$ decides $\sigma$.
Fix a finite increasing sequence $\langle \gamma_0, \cdots, \gamma_{n-1} \rangle \in [A]^n$.
Let $\dot{\nu}_{\vec{\gamma}}=F(\gamma_{n-1})_0$ (Recall that $F(\gamma_{n-1})=(\dot{\nu},\dot{q})$ where $\dot{q}$ is a $P_\gamma$-name of a condition in $P_{\dot{\nu}_{\vec{\gamma}}}/P_{\gamma_{n-1}}$).
For $\xi \in A$, $\xi>\bar{\nu}_{\vec{\gamma}}$, let $\dot{\nu}_\xi=F(\xi)_0$,  $\dot{G}$ be the canonical name of a generic over $P_\xi$ if $\circ(\xi) \leq \max_i \circ(\gamma_i)$, and generic over $P_\xi * P_{\dot{\nu}_\xi}/P_\xi$ otherwise.
Let $s_{\vec{\gamma}}=\stem(p+\vec{\gamma})$ and $s_\xi=(\stem(p+ (\vec{\gamma}^\frown \langle \xi\rangle)))_n$.

By the Prikry property below $\alpha$, let $s^0_{\vec{\gamma},\xi}{}^\frown s^1_{\vec{\gamma},\xi} \leq^* s_{\vec{\gamma}} {}^\frown s_\xi$ be such that $s^0_{\vec{\gamma},\xi}{}^\frown s^1_{\vec{\gamma},\xi}$ decides

\begin{center}

$\exists s \in \dot{G} (s^\frown \langle \alpha, A \setminus \tp(s), F \restriction \{ \gamma \in A\setminus \tp(s) \mid \circ(\gamma)>\circ(s)\} \rangle \parallel \sigma$).

\end{center}

For $\beta<\circ(\alpha)$, let $A_{\vec{\gamma}}(\beta) \in U(\alpha,\beta)$ be such that for each $\xi \in A_{\vec{\gamma}}(\beta)$, $s^0_{\vec{\gamma},\xi}{}^\frown s^1_{\vec{\gamma},\xi}$ gives the same decision, i.e. either for all $\xi \in A_{\vec{\gamma}}(\beta)$,

\begin{center}

$s^0_{\vec{\gamma},\xi}{}^\frown s^1_{\vec{\gamma},\xi} \Vdash \exists s \in \dot{G} (s^\frown \langle \alpha, A \setminus \tp(s), F \restriction \{ \gamma \in A\setminus \tp(s) \mid \circ(\gamma)>\circ(s)\} \rangle \parallel \sigma$),

\end{center}

or for all $\xi \in A_{\vec{\gamma}}(\beta)$

\begin{center}

$s^0_{\vec{\gamma},\xi}{}^\frown s^1_{\vec{\gamma},\xi} \Vdash  \nexists s \in \dot{G} (s^\frown \langle \alpha, A \setminus \tp(s), F \restriction \{ \gamma \in A\setminus \tp(s) \mid \circ(\gamma)>\circ(s)\} \rangle \parallel \sigma$),

\end{center}

If the decision is positive, shrink the set further so that either for all $\xi$,

\begin{center}

 $s^0_{\vec{\gamma},\xi}{}^\frown s^1_{\vec{\gamma},\xi} \Vdash \exists s \in \dot{G} (s^\frown \langle \alpha, A \setminus\tp(s),F \restriction \{ \gamma \in A \setminus \tp(s) \mid \circ(\gamma)>\circ(s)\} \rangle \Vdash \sigma)$,
 
 \end{center}
 
 or for all $\xi$,
 
 \begin{center}

 $s^0_{\vec{\gamma},\xi}{}^\frown s^1_{\vec{\gamma},\xi} \Vdash \exists s \in \dot{G} (s^\frown \langle \alpha, A \setminus\tp(s),F \restriction \{ \gamma \in A \setminus \tp(s) \mid \circ(\gamma)>\circ(s)\} \rangle \Vdash \neg \sigma)$.
 
 \end{center}
 
Let $A^\prime=\triangle_{\vec{\gamma}} \cap_\beta A_{\vec{\gamma}}(\beta)$.
Use Lemma \ref{pigeonhole} with $p=\langle \alpha, A^\prime,F \rangle$ and $s_{\vec{\gamma},\xi}^0 {}^\frown s_{\vec{\gamma},\xi}^1$ to obtain $p^n \leq^*_\alpha p$ such that $\stem(p^n+ (\vec{\gamma}^\frown \langle \xi \rangle)) \leq^* s_{\vec{\gamma},\xi}^0 {}^\frown s_{\vec{\gamma},\xi}^1$.
We now take $p^*$ be such that $p^* \leq^*_\alpha p^n$ for all $n$.

Suppose for a contradiction that there is no direct extension of $p^*$ deciding $\sigma$.
Let $\bar{p} \leq p^*$ with the minimal number of steps such that $\bar{p} \parallel \sigma$.
Assume $\bar{p} \Vdash \sigma$.
Say $\bar{p} \leq p^*$, $l(\bar{p})=n+1$, $\bar{p} \leq p^*+(\vec{\gamma}^\frown \langle \xi \rangle)$ where $|\vec{\gamma}|=n$.
Let $s=\stem(\bar{p})$, $s_0=s \restriction n$ and $s_1=s_n$, so that $s=s_0^\frown s_1$. 
We note that by maximality, $s^\frown \langle \alpha, A \setminus \tp(s), F \restriction \{ \gamma \in A \setminus \tp(s) \mid \circ(\gamma)>\circ(s)\} \rangle \Vdash \sigma$.

\begin{claim}
\label{prikryclaim1}
 $\stem(p^*+ (\vec{\gamma}^\frown \langle\xi \rangle))$ forces that there is $t \in \dot{G}$, $t^\frown \langle \alpha, A^* \setminus \tp(t), F^* \setminus \tp(t) \rangle \Vdash \sigma$.
 
 \end{claim}
 
 \begin{claimproof}{(Claim \ref{prikryclaim1})}
 
Suppose not.
For simplicity, assume that  $\stem(p^*+ (\vec{\gamma}^\frown \langle\xi \rangle))$ forces that there is $t \in \dot{G}$, $t^\frown \langle \alpha, A^* \setminus \tp(t), F^* \setminus \tp(t) \rangle \Vdash \neg \sigma$ (the other case can be done similarly).
Let $G$ be generic containing $\stem{\bar{p}}$, then it contains $\stem(p^*+(\vec{\gamma}^\frown \langle \xi \rangle))$.
Let $s^\prime \in G$ be such that $s^\prime {}^\frown \langle\alpha, A \setminus \tp(s^\prime), F\restriction \{ \gamma \in A \setminus \tp(s^\prime) \mid \circ(\gamma)>\circ(s^\prime)\} \rangle \Vdash \neg \sigma$ (note that $\tp(s^\prime)=\tp(s)$).
We may assume that $s^\prime \leq s$, but then $s^\prime {}^\frown \langle \alpha,A \setminus \tp(s^\prime),F \restriction \{ \gamma \in A \setminus \tp(s^\prime) \mid \circ(\gamma)>\circ(s^\prime)\} \rangle \Vdash \sigma$ and $\neg \sigma$, which is a contradiction.

 \end{claimproof}{(Claim \ref{prikryclaim1})}
 
With a similar proof as in Claim \ref{prikryclaim1}, we can conclude that $p^*+(\vec{\gamma}^\frown \langle \xi \rangle) \Vdash \sigma$, and hence $s_0^\frown \langle \xi,A^* \restriction \xi, F^* \restriction (A^* \restriction \xi), F^*(\xi)_0, F^*(\xi)_1 \rangle^\frown \langle \alpha, A^* \setminus \bar{\nu}_\xi, F^* \restriction \{\tau \mid \tau> \bar{\nu}_\xi$ and $\circ(\tau)>\circ(\xi) \} \rangle \Vdash \sigma$.
This implies that for each $\xi^\prime \in A^*$ with $\circ(\xi^\prime)=\beta$, we have that $s_0^\frown \langle \xi^\prime,A^* \restriction \xi^\prime, F^* \restriction(A^* \restriction \xi^\prime), F^*(\xi^\prime)_0, F^*(\xi^\prime)_1 \rangle^\frown \langle \alpha, A^* \setminus \bar{\nu}_{\xi^\prime}, F^* \restriction \{\tau \mid \tau>\bar{\nu}$ and $\circ(\tau)>\circ(\xi^\prime)\} \rangle \Vdash \sigma$.
Since every extension of $s_0^\frown \langle \alpha,A^* \setminus \tp(s), F^* \restriction \{ \gamma \in A^* \setminus \tp(s) \mid \circ(\gamma)>\circ(s)\} \rangle$ can be extended further to be an extension of $s_0^\frown \langle \alpha,A^* \setminus \tp(s), F^* \restriction \{ \gamma \in A^* \setminus \tp(s) \mid \circ(\gamma)>\circ(s)\} \rangle+\langle \xi^\prime \rangle$ for some $\xi^\prime$ with $\circ(\xi^\prime)=\beta$, and the further extension forces $\sigma$.
By density, $s_0 {}^\frown \langle \alpha, A^* \setminus \tp(s), F^* \restriction \{ \gamma \in A^* \setminus \tp(s) \mid \circ(\gamma)>\circ(s)\} \rangle \Vdash \sigma$, contradicting the minimality of the number of steps used in extension from $p^*$ to $\bar{p}$.

The proof for an impure case is the following: let $s^\frown \langle \alpha,A,F \rangle \in P_\alpha$.
Let $Q$ be the forcing $s$ is living in.
Note that cardinal arithmetic in $V$ is the same as in $V[G]$.
Let $G$ be $Q$-generic, then work in $P_\alpha[G]$ with $(\langle \alpha,A,F \rangle)[G]$ in a similarly to before to get $\langle \alpha,B,H \rangle$ deciding $\sigma[G]$ (note that Proposition \ref{measureoneinforcing} is necessary here, since when we split a measure-one set $A \in \vec{U}(\alpha)$ into disjoint sets, and one of them is of measure-one, then that set may have measure one in the ultrafilters in the extension, however, we can shrink the set to have measure one with respect to ultrafilters in the ground model).
Back in $V$, by the Prikry property of $Q$, let $s^* \leq^* s$ such that $s^* \Vdash_Q \langle \alpha,\dot{B},\dot{H} \rangle \Vdash_{P_\alpha/Q} \sigma^i$ for some $i=0,1$, where $\sigma^0=\sigma$ and $\sigma^1=\neg \sigma$.
By Proposition \ref{measureinground}, we can shrink $\dot{B}$ to $B \in V$ such that $\Vdash_Q \check{B} \subseteq \dot{B}$.
Then find $H$ such that for all $\gamma$, $H(\gamma)_1$ is forced to be $\leq^*$-stronger than $\dot{H}(\gamma)_1$,
Hence, $s^* {}^\frown \langle \alpha,B,H \rangle$ decides $\sigma$ as required.

\end{proof}

With a similar proof of the proof for the Prikry property, we obtain the following.

\begin{prop}

Let $\beta,\dot{\nu}$, and $Q$ be as in Definition \ref{quotient}, then $\Vdash_Q (P_\alpha/Q,\leq,\leq^*)$ has the Prikry property.

\end{prop}

\begin{proof}

First, suppose that pure conditions in $P_\alpha/Q$ are defined.
Let $G$ be $Q$-generic and in $V[G]$, let $p \in P_\alpha[G]$ is pure.
Then proceed as in Proposition \ref{mainprikry}.
Now, suppose either impure conditions in $P_\alpha/Q$ are defined, and in $V[G]$, $p \in P_\alpha[G]$ is impure, or impure conditions in $P_\alpha/Q$ are not defined and $p \in P_\alpha[G]$ be any conditions.
Then $p=s^\frown \langle \alpha,A,F \rangle$.
Proceed the proof of the Prikry property as in Proposition \ref{mainforcing} for impure conditions.

\end{proof}

\section{The cardinal behavior in $P_\alpha$}
\label{cardinalbehave}
By a chain condition of $P_\alpha$, cardinals above $\alpha^+$, exclusively, are preserved, and the forcing preserves cofinalities of regular cardinals above $\alpha^+$
We first show that cardinals up to, but not including $\alpha$, are preserved.

\begin{prop}
\label{presbelow}
All cardinals up to and including $\alpha$, are preserved.

\end{prop}

\begin{proof}

We first show that cardinals below $\alpha$ are preserved.
Consider each cardinal $\gamma<\alpha$ and $p \in P_\alpha$.
Set $p=s^\frown \langle \alpha,A,F \rangle$.
Let $G$ be generic containing $s$.
Since it is from the forcing below $\alpha$, by induction, $\gamma$ is a cardinal in $V[G]$.
Let $\bar{p}=(\langle \alpha,A,F \rangle)[G]$.
By the Prikry property and the fact that the direct extension relation for pure conditions in $P_\alpha[G]$ is $\alpha$-closed for arbitrary $G$, we get that $p \Vdash_\alpha \gamma$ is preserved.
Finally, note that since $\alpha$ is strongly inaccessible in $V$, $\alpha$ is preserved in the extension.

\end{proof}

\begin{prop}
\label{GCHbelow}
GCH below $\alpha$ holds in $V^{P_\alpha}$.

\end{prop}

\begin{proof}

Let $\gamma<\alpha$.
Let $p=s^\frown \langle \alpha,A,F\rangle$.
Forcing with $p$ is factored into $P^s$, the forcing containing $s$, and $P_\alpha/P^s$.
Since $P^s$ is a forcing below $\alpha$,$\Vdash_{P^s} (P_\alpha/P^s,\leq,\leq^*)$ has the Prikry property, and $\Vdash (P_\alpha/P^s,\leq^*)$ is $\alpha$-closed, forcing with $p$ preserves GCH at $\gamma$.

\end{proof}

By Propositions \ref{presbelow} and \ref{GCHbelow}, 

\begin{coll}

$\alpha$ is preserved (as a cardinal).

\end{coll}
If $\circ(\alpha)=0$, $P_\alpha$ is equivalent to $P_\beta *P_{\dot{\nu}}/P_\beta$ for some $\beta$ and $\dot{\nu}$, which is forced to be below $\alpha$, hence, $\alpha$ is still inaccessible.
If $\circ(\alpha)>0$, the forcing $P_\alpha$ clearly projects down to a Prikry forcing or a Magidor forcing, adding a  club subset of $\alpha$ of order-type $\omega^{\circ(\alpha)}$ (the ordinal exponentiation).
If $\circ(\alpha)=0$, then obviously $\alpha^+$ is still regular in $V^{P_\alpha}$.

\begin{prop}

If $\circ(\alpha)>0$, then $\alpha^+$ is preserved and remains regular in $V[G_\alpha]$.

\end{prop}

\begin{proof}

Since $0<\circ(\alpha)<\alpha$, in $V[G_\alpha]$, $\cf(\alpha)<\alpha$.
Suppose $\alpha^+$ is collapsed, then there is $\delta<\alpha$ and a cofinal map $f: \delta \to (\alpha^+)^V$.
Let $\dot{f}$ be a such name, forced by $p$.
Assume for simplicity that $p=\langle \alpha,A,F \rangle$ (if $p$ is not pure, then we may first force over the stem of $p$ and proceed as usual).
Since $\delta<\alpha$, we may find $p^\prime \leq_\alpha^* p$ such that  if $p^\prime=\langle \alpha, A^\prime, F^\prime \rangle$, then for all $\gamma<\delta$, if $s^\frown \langle \alpha, B,H \rangle \leq_\alpha p^\prime$ and $s^\frown \langle \alpha,B,H \rangle$ decides $\dot{f}(\gamma)$, then so is $s^\frown \langle \alpha, A^\prime \setminus \tp(s), F^\prime \restriction \{ \gamma \in A^\prime \setminus \tp(s) \mid \circ(\gamma)>\circ(s)\} \rangle$.
Let $X=\{\xi \mid \exists \gamma<\delta \exists \text{stem } s (s^\frown \langle \alpha, A^\prime \setminus \tp(s), F^\prime\restriction \{ \gamma \in A^\prime \setminus \tp(s) \mid \circ(\gamma)>\circ(s)\} \rangle \Vdash_\alpha \dot{f}(\check{\gamma})=\check{\xi})\}$.
Then $p^\prime \Vdash_\alpha \rge(\dot{f}) \subseteq \check{X}$ and $|X|\leq \alpha \otimes \delta=\alpha$.
Hence, $p^\prime \Vdash_\alpha \rge(\dot{f})$ is bounded in $(\alpha^+)^V$, which is a contradiction.

A similar proof shows that the cofinality of $\alpha^+$ is preserved, so $\alpha^+$ is regular.

\end{proof}

Since $|P_\alpha| \leq \alpha^{++}$, we have that GCH hold everywhere except possibly at $\alpha$.
The next proposition shows that GCH holds at $\alpha$.

\begin{prop}

GCH holds at $\alpha$.

\end{prop}

\begin{proof}

Suppose for a contradiction that in $V^{P_\alpha}$, $2^\alpha \geq \alpha^{++}$.
Let $p \Vdash_\alpha \{\dot{f}_\gamma \mid \dot{f}_\gamma: \alpha \to 2, \gamma<\alpha^{++}\}$ a sequence of different subsets of $\alpha$.
Since forcing with stem of $p$ does not change GCH at $\alpha$, we may assume for simplicity that $p$ is pure.
Let $D_{\gamma,\xi}$ be the collection of $p^\prime$ that decides $\dot{f}_\gamma(\xi)$.
Use Proposition \ref{maximaltop} to get a $\leq^*$-decreasing sequence $ \vec{p}_\gamma=\{p_{\gamma,\xi} \mid \xi<\alpha\}$ each $p_{\gamma,\xi}$ witnesses Proposition \ref{maximaltop} for $D_{\gamma,\xi}$.
By GCH in $V$, $(\alpha^+)^\alpha=\alpha$, we can shrink the family of functions so that there is $\vec{p}$ with $\vec{p}_\gamma=\vec{p}$ for all $\gamma$.
Write $\vec{p}=\{ p_\gamma \mid \gamma<\alpha\}$, $p_\gamma=\langle \alpha,A_\gamma,F_\gamma \rangle$.
Note that the number of conditions that decide $\dot{f}_\gamma(\xi)$ for some $\gamma,\xi$ is at most $\alpha$: i.e. the collection of $s^\frown \langle \alpha, A_\gamma \setminus \tp(s),F_\gamma \restriction \{ \gamma \in A_\gamma \setminus \tp(s) \mid \circ(\gamma)>\circ(s)\}\rangle$ for some $s,\gamma$.
Then, we count the number of nice names for functions from $\alpha$ to $2$, with respect to conditions in $\{s^\frown \langle \alpha, A_\gamma \setminus \tp(s),F_\gamma \restriction \{ \gamma \in A_\gamma \setminus \tp(s) \mid \circ(\gamma)>\circ(s)\} \rangle\mid$ for some $s$ and $\gamma \}$.
It is at most $\alpha^{\alpha \otimes \alpha}=\alpha^+$, which is a contradiction.

\end{proof}

Let $G$ be $P_\alpha$-generic.
If $\circ(\alpha)=0$, then $P_\alpha$ is isomorphic to $Q$ for some $\beta,\dot{\nu}$ and $Q$ as in Definition \ref{quotient}.
Define $C_\alpha$ in $V[G]$ by $C_\alpha=C_Q \cup \{\alpha\}$ where $C_Q$ is the set added by $Q$ (if $Q=P_\beta$, then $C_Q=C_\beta$, otherwise, $C_Q=C_\beta \sqcup C_{\nu/\beta}$ where $\nu=\dot{\nu}[G \restriction P_\beta]$.
Then $C_\alpha \subseteq \alpha+1$, $\max(C_\alpha)=\alpha$, and $C_\alpha \setminus \{\alpha\}$ is a closed bounded subset of $\alpha$.

Consider the general case, $\circ(\alpha)>0$.
Let $G$ be $P_\alpha$-generic.
Let $C^*=\{\gamma<\alpha \mid \exists p \in P_\alpha \exists i (\stem(p)_i)_0=\gamma\}$.
The tail of $C^*$ is simply a Prikry sequence if $\circ(\alpha)=1$, or a magidor  club set of $\alpha$ if $\circ(\alpha)>1$.
For $\gamma \in C^*$, let $G_\gamma=G \restriction P_\gamma$.
For such $\gamma$.
For $p \in G$ with $(\stem(p)_i)_0=\gamma$ for some $i$, let $s_{\gamma,p}=\stem(p)_i$.
If $s_{\gamma,p}$ is a triple, then $s_{\gamma,p^\prime}$ is a triple for all $p^\prime \in G$  with $s_{\gamma,p^\prime}$ exists. In this case, let $C^*_\gamma=\emptyset$.
If, otherwise, $s_{\gamma,p}$ is a $5$-tuple.
Set $R=\{\dot{q}[G_\gamma] \mid \exists p^\prime \in G((s_{\gamma,p^\prime})_4=\dot{q}) \}$.
Then, $R$ is $P_{\dot{\nu}[G_\gamma]}[G_\gamma]$-generic over $V[G_\gamma]$ for some $\dot{\nu}$.
This produces the set $C_{\nu_\gamma/\gamma} \subseteq (\gamma,\nu_\gamma]$  where $\nu_\gamma=\dot{\nu}[G_\gamma]$.
We have that $\nu_\gamma \in [\gamma,\max(C^* \setminus (\gamma+1))$.
If $\nu_\gamma=\gamma$, then $C_{\nu_\gamma/\gamma}=\emptyset$.
Assume $\nu_\gamma>\gamma$.
Then $\max(C_{\nu_\gamma/\gamma})=\nu_\gamma$, $C_{\nu/\gamma} \setminus \{\nu\}$ is a  club set of $\nu$ if $\circ(\nu)>0$, or a closed bounded subset of $\nu$ if $\circ(\nu)=0$.
Furthermore, $\nu<\min(C^* \setminus (\nu+1))$.
Set $C_\alpha=C^* \cup \cup\{C_{\nu_\gamma/\gamma} \mid \gamma \in C^*\} \cup \{\alpha\}$.

Now, if $\beta,\dot{\nu}$ and $Q$ be as in Definition \ref{quotient}, and $G$ is $Q$-generic.
In $P_\alpha[G]$, we can apply the same process to get the required set $C_\alpha/Q$ as follows.
Let $H$ be $P_\alpha[G]$-generic, and $I$ be the canonical generic over $P_\alpha$ such that $V[G][H]=V[I]$.
In $V[I]$, we have $C_\alpha$.
Set $C_{\alpha/Q}=C_\alpha \setminus C_Q$.

\begin{prop}

Forcing with $P_\alpha$ preserves cofinalities of regular cardinals which are not in $\lim(C_\alpha)$.

\end{prop}

\begin{proof}

Let $\gamma \not \in \lim(C_\alpha)$.
First, suppose that $\circ(\alpha)=0$, then $P_\alpha$ is equivalent to some $P_\beta * P_{\dot{\nu}}/P_\beta$, and $C_\alpha=C_\beta \sqcup C_{\nu/\beta} \cup \{\alpha\}$ for some $\nu$.
If $\gamma=\alpha$, then $\gamma$ is regular in $V[G]$ since $P_\alpha$ is equivalent to a small forcing.
If $\gamma<\alpha$, then $\gamma \not \in \lim(C_\beta \sqcup C_{\nu/\beta})$, and so $\gamma$ is regular in $V[G \restriction P_\beta * P_{\dot{\nu}}/P_\beta]=V[G]$.
Now, suppose that $\circ(\alpha)>0$.
Then find $\beta \in (\gamma,\alpha)$ such that $G \restriction P_\beta$ exists.
Note that $\gamma \not \in \lim(C_\beta)$, so $\gamma$ is regular in $V[G \restriction P_\beta]$.
The forcing $(P_\alpha[G \restriction P_\beta], \leq^*)$ is $\beta^+$-closed, so $\gamma$ is still regular in $V[G]$.

\end{proof}

Similar proof with an easy application on factoring forcings show that

\begin{prop}

$\Vdash_\alpha \gamma \not \in \lim(C_{\alpha/Q})$ and $(\gamma$ is regular)$^{V^Q}$ implies that $\gamma$ is regular.

\end{prop}

\section{The main forcing $\mathbb{P}_\kappa$}
\label{mainforcing}

We now define the main forcing $\mathbb{P}_\kappa$.

\begin{defn}

Define $\mathbb{P}_\kappa=\cup_{\alpha<\kappa} P_\alpha$.
For $p,q \in \mathbb{P}_\kappa$, define $p\leq q$ if $p,q \in P_\alpha$ and $p \leq_\alpha q$, or $p \in P_\beta$, $q \in P_\alpha$, $\beta>\alpha$, and $p \restriction P_\alpha \leq_\alpha q$.

\end{defn}

If $G$ is $\mathbb{P}_\kappa$-generic, then for $\alpha<\kappa$ with some $p \in G \cap P_\alpha$, $G_\alpha:= G \restriction P_\alpha$ is $P_\alpha$-generic.
$\mathbb{P}_\kappa$ has size $\kappa$, so it is $\kappa^+$-c.c.

\begin{thm}

Let $G$ be $\mathbb{P}_\kappa$-generic.
In $V[G]$, if $f: \alpha \to On$ for some $\alpha<\kappa$, then there is $\beta<\kappa$ such that $f \in V[G \restriction P_\beta]$.

\end{thm}

\begin{proof}

Recall that for $\gamma<\kappa$, $\{\alpha \in X \mid \circ(\alpha)\geq \gamma\}$ is stationary.
Let $\dot{f}: \alpha \to On$ and $p \in \mathbb{P}_\kappa$.
Let $M \prec H_\theta$ such that $M \cap \kappa \in \kappa$, $M$ is transitive below $M \cap \kappa$, $\mathbb{P}_\kappa,p,\kappa,\dot{f} \in M$, and if $\gamma:=M \cap \kappa$, then $\circ(\gamma)>\alpha$ is a regular cardinal and $\circ(\gamma)$ is greater than $\circ(\xi)$ any $\xi$ which is equal to $(\stem(p)_i)_0$ for some $i$, say that $\xi_0$ is  the maximum of those $\circ(\xi)$s.
Let $A \in \vec{U}(\gamma)$.
We build $F$ with $\dom(F)=\{\beta \in A \mid \circ(\beta)>\xi_0\}$.
We define the value $F$ inductively on the Mitchell order.
For each $\beta$ with $\circ(\beta)=\xi_0+1$, let $G_\beta$ be the forcing containing $p^\frown \langle \gamma, A \restriction \beta, \emptyset \rangle$.
Then, in $V[G_\beta]$, there is $q_\beta \in \mathbb{P}_\kappa/G_\beta$ such that $q_\beta$ decides $\dot{f}[G_\beta](0)$.
Since $M \cap \kappa$ is transitive and $\beta<M \cap \kappa$, we may find $q_\beta \in M[G_\beta]$.
Note that $\kappa \cap M[G_\beta]=\kappa \cap M$.
Thus, $q_\beta \in P_\nu/P_\beta$ for some $\nu_\beta<\gamma$.
Let $\dot{\nu}_\beta$ and $\dot{q}_\beta$ be names for such $ \nu_\beta$ and $q_\beta$ and define $F(\beta)=\langle \dot{\nu}_\beta,\dot{q}_\beta \rangle$.
Suppose that $F$ is defined on ordinals of Mitchell orders in $(\xi_0,\xi_1)$ for some $\xi_1<\xi_0+\alpha$.
Write $\xi_1=\xi_0+\tau$ for unique $\tau$.
For $\beta \in A$ with $\circ(\beta)=\xi_1$, let $G_\beta$ be generic containing $p^\frown \langle \alpha,A \restriction \beta, F \restriction \{\tau \in A \restriction \beta \mid \circ(\tau)>\xi_0\} \rangle$.
Find $q_\beta \in \mathbb{P}_\kappa/G_\beta \cap M[G_\beta]$ deciding $\dot{f}[G_\beta](\tau)$, so $q_\beta \in P_\nu[G_\beta]$,  and let $\dot{\nu}_\beta$, $\dot{q}_\beta$ be such names.
Set $F(\beta)=\langle \nu_\beta,\dot{q}_\beta\rangle$.
This process continues up to (but not including) stage $\xi_0+\alpha$.
For $\beta$ with $\circ(\beta) \geq \xi_0+\alpha$, assign any $F(\beta)$.
We can see that by the choice of $F$, if $G$ is $\mathbb{P}_\kappa$-generic containing $p^\frown \langle \alpha,A,F \rangle$, then $\dot{f} \in V[G \restriction P_\gamma]$.
Thus, $p^\frown \langle \gamma,A,F \rangle\Vdash_\kappa \dot{f} \in V[\dot{G} \restriction P_\gamma]$.
\end{proof}

\begin{coll}

In $V^{\mathbb{P}_\kappa}$, $\kappa$ is still inaccessible, and GCH holds below $\kappa$.

\end{coll}

Let $G$ be $\mathbb{P}_\kappa$-generic.
Set $C_\kappa=\cup_{\alpha<\kappa} C_\alpha$ where $C_\alpha$ is the set derived from $V[G \restriction P_\alpha]$.
For $\alpha<\beta<\kappa$, $C_\alpha \sqsubseteq C_\beta$.

\begin{prop}

$C_\kappa$ is a club subset of $\kappa$.

\end{prop}

\begin{proof}

Since for $p \in P_\alpha$ and $\beta \in X \setminus (\alpha+1)$, there is $p^\prime \leq_\kappa p$, $p^\prime \in P_\beta$, $C_\kappa$ is unbounded in $\kappa$.
To check the closeness, let $\gamma$ be the limit point of $C_\kappa$, then it is forced by some $p \in P_\alpha$ for some $\alpha>\gamma$.
We see that $\gamma$ is a limit point in $C_\alpha$.
Thus, $\gamma$ is a limit point of $C_\kappa$.

\end{proof}

\begin{prop}

If $\gamma \not \in \lim(C_\kappa)$ is regular in $V$, then in is regular in $V^{\mathbb{P}_\kappa}$.

\end{prop}

\begin{proof}

Let $\gamma \not \in \lim(C_\kappa)$ be regular in $V$.
Suppose for a contradiction that $\beta<\gamma$ and $f: \beta \to \kappa$ be cofinal.
Find $\alpha<\kappa$ large such that $f \in V[G \restriction P_\alpha]$.
Since $\gamma \not \in \lim(C_\kappa)$, $\gamma \not \in \lim(C_\alpha)$, so $\gamma$ is regular in $V[G \restriction P_\alpha]$, which is a contradiction.

\end{proof}

\section{The optimal assumption}
\label{optimal}

As we noted earlier, the following assumption is also enough to build the forcing that has the same effect:

\begin{assumpt}
\label{optimalassumpt}
$\kappa$ is inaccessible, $X \subseteq \kappa$, each $\alpha \in X$, $\{U(\alpha,\beta) \mid \beta<\circ(\alpha)\}$ is a Mitchell increasing sequence of normal measures, and for $\nu<\kappa$, $\{\alpha \in X \mid \circ(\alpha) \geq \nu\}$ is stationary.

\end{assumpt}

We complete this section by recalling, with proof, that Assumption \ref{optimalassumpt} is optimal.

\begin{prop}

Assume there is no inner model of $0^\P$.
Let $V \subseteq W$.
Suppose that in $W$, there is a club subset $C$ of $\kappa$ containing $V$-regular cardinals, $\kappa$ is inaccessible in $W$, and all $V$-regular cardinals which are not limit points of $C$ are still regular in $W$.
Then, $V$ satisfies Assumption \ref{optimalassumpt}.

\end{prop}

\begin{proof}

Let $\nu<\kappa$ and $C^\prime \subseteq \kappa$ be a club in $V$.
By absoluteness, $C^\prime$ is a club in $W$, and so $C^\prime \cap C$ is a club.
Find $\gamma \in C^\prime \cap C$ such that $\cf^W(\gamma) \geq \nu$.
Since $\gamma$ is regular in $V$, by \cite{mitchell1987applicationsofthecoremodel}, $\circ(\gamma) \geq \nu$ in $V$.

\end{proof}

\appendix

\bibliographystyle{ieeetr}
\bibliography{addingclubfromoptimalassumptionreferences}
\end{document}